\documentclass[12pt,a4paper]{amsart}

\usepackage[utf8]{inputenc}
\usepackage[T1]{fontenc}
\usepackage{amsmath,amssymb,amsthm,mathtools,mathrsfs}
\usepackage{enumitem}
\usepackage{tikz-cd}
\usepackage{booktabs}
\usepackage[colorlinks=true,citecolor=blue,linkcolor=blue,urlcolor=blue]{hyperref}
\usepackage{aliascnt}
\usepackage[capitalise]{cleveref}
\usepackage{microtype}

\newtheorem{theorem}{Theorem}[section]

\newaliascnt{proposition}{theorem}
\newtheorem{proposition}[proposition]{Proposition}
\aliascntresetthe{proposition}

\newaliascnt{corollary}{theorem}
\newtheorem{corollary}[corollary]{Corollary}
\aliascntresetthe{corollary}

\newaliascnt{lemma}{theorem}
\newtheorem{lemma}[lemma]{Lemma}
\aliascntresetthe{lemma}

\theoremstyle{definition}

\newaliascnt{definition}{theorem}
\newtheorem{definition}[definition]{Definition}
\aliascntresetthe{definition}

\newaliascnt{example}{theorem}
\newtheorem{example}[example]{Example}
\aliascntresetthe{example}

\newaliascnt{model}{theorem}

\aliascntresetthe{model}

\theoremstyle{remark}

\newaliascnt{remark}{theorem}
\newtheorem{remark}[remark]{Remark}
\aliascntresetthe{remark}

\crefname{remark}{Remark}{Remarks}
\Crefname{remark}{Remark}{Remarks}

\crefname{theorem}{Theorem}{Theorems}
\Crefname{theorem}{Theorem}{Theorems}

\crefname{proposition}{Proposition}{Propositions}
\Crefname{proposition}{Proposition}{Propositions}

\crefname{corollary}{Corollary}{Corollaries}
\Crefname{corollary}{Corollary}{Corollaries}

\crefname{lemma}{Lemma}{Lemmata}
\Crefname{lemma}{Lemma}{Lemmata}

\crefname{definition}{Definition}{Definitions}
\Crefname{definition}{Definition}{Definitions}

\crefname{example}{Example}{Examples}
\Crefname{example}{Example}{Examples}

\crefname{model}{Model}{Models}
\Crefname{model}{Model}{Models}

\crefformat{equation}{(#2#1#3)}
\Crefformat{equation}{Equation~(#2#1#3)}
\crefmultiformat{equation}{(#2#1#3)}{ and~(#2#1#3)}{, (#2#1#3)}{ and~(#2#1#3)}
\crefrangeformat{equation}{(#3#1#4) to~(#5#2#6)}

\newcommand{\R}{\mathbb{R}}
\newcommand{\N}{\mathbb{N}}

\newcommand{\Z}{\mathbb{Z}}
\newcommand{\LL}{\mathcal{L}}        

\newcommand{\EE}{\mathcal{E}}        

\newcommand{\norm}[1]{\lVert #1 \rVert}

\newcommand{\id}{\mathrm{id}}

\begin{document}

\title[Linearization on Banach bundles]{A Hartman--Grobman theorem\\ on Banach bundles}

\author[Rogoll]{Benjamin Rogoll}
\address{Institut f\"ur Analysis, Fakult\"at Mathematik,
  TU Dresden, 01062 Dresden, Germany}
\email{benjamin.rogoll@tu-dresden.de}

\author[Siegmund]{Stefan Siegmund}
\address{Institut f\"ur Analysis, Fakult\"at Mathematik,
  TU Dresden, 01062 Dresden, Germany}
\email{stefan.siegmund@tu-dresden.de}

\date{\today}

\keywords{Hartman--Grobman theorem, skew-product flow, Banach bundle,
  exponential dichotomy, nonuniform exponential dichotomy,
  H\"older regularity, control-affine system}

\subjclass[2020]{37B55, 34D09, 37C15, 37D25, 93C10}

\begin{abstract}
We provide generalized linearization theorems for skew-product flows on abstract Banach bundles over locally compact metric spaces, allowing for a unified treatment of uniform and nonuniform, exponential and strong exponential dichotomies. This is accomplished by constructing an operator whose fixed points are exactly the conjugacies that are bounded perturbations of the identity, and by giving criteria under which this operator is a contraction. Under stronger assumptions we refine the argument to obtain H\"older regularity of the conjugacy. Finally, we apply our results to recover known results for non-autonomous differential equations and to obtain new linearization results for control-affine systems on manifolds.
\end{abstract}
\maketitle
\section{Introduction}\label{sec:intro}

When does the linearization of a dynamical system behave topologically like the original system? Classically such a question has been answered by a myriad of connected results. The classical Hartman--Grobman theorem answers this question for autonomous ODEs near a hyperbolic  equilibrium~\cite{Hartman1960LemmaTheory,Grobman1959HomeomorphismsSystems}. Three successive generalizations have shaped the modern theory:

\smallskip
\noindent\emph{(i) From autonomous to nonautonomous.}\;
Palmer~\cite{Palmer1973GeneralizationHartmans} replaced the hyperbolic matrix by a linear system $\dot x = A(t)\,x$ on $\R^N$ with an \emph{exponential dichotomy} (ED). The dichotomy concept itself goes back to Perron and was reshaped for nonautonomous problems by Sacker and Sell, who in a parallel line of work established existence and roughness of dichotomies and invariant splittings for linear skew-product flows~\cite{SackerSell1974ExistenceDichotomies,SackerSell1976ExistenceDichotomies}. The H\"older regularity of the Hartman--Grobman conjugacy in the autonomous finite-dimensional case has a more intricate history, with first proofs going back to Belitskii~\cite{Belitskii1973FunctionalEquations} in the 1970s; in the modern formulation we use, the H\"older estimate is part of the systematic treatment by Barreira and Valls~\cite{BarreiraValls2011SimpleProof}. Further explorations of this line by Backes, Dragi\v{c}evi\'c, and Zhang~\cite{BackesEtAl2022SmoothLinearization} extend smooth linearization to polynomial (non-exponential) behavior, while Backes and Dragi\v{c}evi\'c~\cite{BackesDragicevic2023MultiscaleLinearization} develop a multiscale variant that allows different rate controls along complementary directions.  In the discrete setting, Casta\~neda, Gonz\'alez, and Robledo~\cite{CastanedaEtAl2021TopologicalEquivalence} establish topological equivalence for nonautonomous difference equations with a family of dichotomies on the half line. 

\smallskip
\noindent\emph{(ii) From uniform to nonuniform.}\;
Barreira and Valls~\cite{BarreiraValls2011SimpleProof} extended Palmer's approach to \emph{nonuniform} exponential dichotomies, where the dichotomy estimates may deteriorate along orbits at a small exponential rate, the natural setting for smooth ergodic theory~\cite{BarreiraPesin2007NonuniformHyperbolicitya}. For random dynamical systems the corresponding linearization along hyperbolic stationary trajectories is due to Coayla-Terán, Mohammed, and Ruffino~\cite{Coayla-TeranEtAl2007HartmanGrobman}.

\smallskip
\noindent\emph{(iii) From finite to infinite dimensions.}\;
Pugh~\cite{Pugh1969TheoremHartman} extended Hartman--Grobman to arbitrary Banach spaces and Dragi\v{c}evi\'c et al.~\cite{DragicevicEtAl2025GlobalLinearization} have recently given a self-contained reworking of this result, and Backes and Dragi\v{c}evi\'c~\cite{backes_stability_2024} consider nonautonomous systems on Fréchet spaces; Aulbach and Wanner~\cite{AulbachWanner2000HartmanGrobman} proved a Hartman--Grobman theorem for Carath\'eodory-type equations on Banach spaces, and Hein and Pr\"uss~\cite{HeinPruss2016HartmanGrobman} obtained a H\"older version for autonomous semilinear equations \(\dot v = Av + r(v)\) when $A$ generates a $C_0$-\emph{group}. Bates and Lu~\cite{BatesLu1994HartmanGrobman} treat autonomous evolution equations in a Hilbert space carrying an inertial manifold. A more abstract strand obtains the Banach-space theorem through a topological \emph{decoupling}, or reduction principle, valid in arbitrary complete metric spaces~\cite{Reinfelds1994ReductionPrinciple}, while $C^1$ linearizations on Banach spaces are due to Rodrigues and Sol\`a-Morales~\cite{rodrigues_smooth_2004}. Pötzsche and Russ~\cite{PotzscheRuss2016TopologicalDecoupling} give a non-autonomous counterpart of this decoupling strand, with a spectral gap condition on the Sacker--Sell spectrum, replacing invertibility. 

\smallskip
All three lines of generalization work on a \emph{fixed} state space~$X$, usually a Banach space, the conjugacy being a family of homeomorphisms $H_t \colon X \to X$. This paper develops a Hartman--Grobman theory directly on Banach bundles, providing a single framework for trivial and non-trivial bundles, finite- and infinite-dimensional fibers, uniform and nonuniform dichotomies, with an explicit H\"older exponent matching the Barreira--Valls exponent to leading order~\cite{BarreiraValls2011SimpleProof}.  When specialized to trivial bundles, the results recover the classical theorems of Palmer, Barreira--Valls, and  Aulbach--Wanner as special cases.

\subsection{Structure of the paper}
We begin this paper by giving a loose introduction to the theory of Banach bundles, as well as stating some necessary facts about their topology over locally compact metric spaces and their sections. We finish the second section by introducing skew-product flows as a notion of dynamical system, which generalizes the notion of a two-parameter family of solutions to differential equations, and with that the notion of semigroups. A particularly interesting thing here is that there is no longer a singular notion of a ``state space'', but rather what is considered a state might be subject to change with the dynamic of the base space. It should be noted that, since all our dynamics are reversible, along a given trajectory in the base space all ``state spaces'' will be topologically isomorphic, but still might carry different geometries. This is intended, as it enabled us to give a unified treatment of uniform and nonuniform behaviors.

In \cref{sec:ED} we give a generalized notion of exponential dichotomy, uniform and nonuniform, as well as strong exponential dichotomy, also uniform and nonuniform. We furthermore show that the discrepancy between uniform and nonuniform notions is not intrinsic to the dynamics, but rather a consequence of the ambient geometry where the dynamic is studied. This means we are showing that, using a construction reminiscent of Lyapunov norms, we can, to a given nonuniform dichotomy, find a topologically equivalent Banach bundle structure under which the dichotomy becomes uniform.

Continuing to \cref{sec:HGT} we prove a series of Hartman--Grobman type theorems for uniform, or nonuniform, exponential dichotomies and strong exponential dichotomies. We accomplish this using a similar strategy to~\cite{BarreiraValls2011SimpleProof}, namely by realizing a conjugation as the existence of a fixed point of a certain operator, which our assumptions guarantee to be a contraction on a Banach space. This technique can be further adapted to obtain local and Hölder analogs of our theorem.

Finally, in \cref{sec:App} we apply the proven theorems to obtain a new result in control theory which under assumptions of finite-dimensionality allows for the linearization of control-affine systems in a topologically equivalent way. We do this by viewing the underlying set of control functions as a compact metric space when equipped with the weak-* topology, following the work of Kawan~\cite{Kawan2013InvarianceEntropy,Kawan2017UniformlyHyperbolic}, and the control-affine system as a skew-product flow over the shift on said space. The closest existing result is the abstract Hartman--Grobman theorem of Baratchart, Chyba, and Pomet~\cite{BaratchartEtAl2007GrobmanHartmana}, which already conjugates over a topological space of inputs, but only allows for a \emph{constant} linear part and possibly non-autonomous non-linearities, so the dichotomy subspaces do not vary in time.

Lastly we show that our approach is strong enough to recover and extend known linearization results by applying our theorems to recover results of Aulbach and Wanner~\cite{AulbachWanner2000HartmanGrobman} as a representative example.

\subsection{Motivation and origins of the bundle formulation}

In this paper we chose a fairly unorthodox way to approach the subject of linearization in dynamics. This has not been the goal from the start. At first, we set out to prove a Hartman--Grobman theorem for the more classically known class of skew-product flows where the bundle structure is trivial, so just given by a cartesian product \( B \times X\). This approach was sufficient and elegant up until nonuniform exponential dichotomy sparked our interest. There, the definitions were ill-fitted to the setting of skew-product flows. We fixed this problem by understanding nonuniform dichotomies as uniform dichotomies along an adapted family of norms as has been explored by Barreira and Valls~\cite{BarreiraValls2011SimpleProof} and by Barreira, Dragi\v{c}evi\'c, and Valls~\cite{BarreiraEtAl2018AdmissibilityHyperbolicity}. This has provided reason enough to generalize the theory to settings which are no longer trivial, since even when one is starting out on a trivial bundle \( B \times X\) it is no longer wise to see the new space, which has been renormed over every \( b \in B\), as trivial. And due to recent works by Kreidler and Siewert~\cite{kreidler_gelfand-type_2020} and Siewert~\cite{Siewert2020WeightedKoopman} we have found Banach bundles to be a suitable setting to develop the theory on.

\section{Banach bundles, skew-product flows, and conjugacies}
  \label{sec:framework}

\subsection{Banach bundles}\label{subsec:BB}

Throughout the paper we fix a \emph{locally compact metric space} \(B\) (the \emph{base}).  A Banach bundle is, intuitively, a continuously varying family of Banach spaces \(\{E_b\}_{b \in B}\) glued together into a single topological space~\(\EE\).  The classical examples are tangent bundles \( TM\) of a manifold \( M\) with an additional structure that turns every tangent space into a Banach space in a compatible manner. This is the case, for example, with symmetric Finsler manifolds or Riemannian manifolds. In our setting we need two generalizations: the fibers may be infinite-dimensional Banach spaces, and the bundle need \emph{not} be locally trivial.  The minimal axioms compatible with these requirements were isolated by Fell and Doran in the context of \(C^*\)-algebraic representation theory and consist of four conditions on the fiber norm, addition and scalar multiplication, and the topology near the zero section.

\begin{definition}[Banach bundle]\label{def:BanachBundle}
  A \textit{Banach bundle over \( B\)} is a triple
\((\EE, \pi, B)\) where \(\EE\) is a Hausdorff topological space
(the \emph{total space}),
\(\pi \colon \EE \to B\) is a continuous \emph{open} surjection,
and each fiber \(E_b := \pi^{-1}(b)\) is a Banach space with norm
\(\norm{\cdot}_b\), subject to:
\begin{enumerate}[label=(BB\arabic*)]
  \item \label{ax:norm}
    The map \(v \mapsto \norm{v}_{\pi(v)}\) is continuous on~\(\EE\).
  \item \label{ax:add}
    The addition map
    \(\EE \times_B \EE \to \EE\), \((u,v) \mapsto u + v\)
    (on the fiber product
    \(\EE \times_B \EE = \{(u,v) \in \EE^2 : \pi(u) = \pi(v)\}\))
    is continuous.
  \item \label{ax:scal}
    The scalar multiplication \(\R \times \EE \to \EE\),
    \((\lambda, v) \mapsto \lambda v\), is continuous.
  \item \label{ax:nbhd}
    For every \(b \in B\) and every open \(W \subseteq \EE\) containing
    \(0_b \in E_b\), there exist \(\varepsilon > 0\) and a neighborhood
    \(U\) of \(b\) such that
    \(\{v \in \pi^{-1}(U) : \norm{v} < \varepsilon\} \subseteq W\).
\end{enumerate}
\end{definition}
Our standing topological assumptions (\(B\) locally compact metric, hence in particular locally compact Hausdorff and first countable; \(\EE\) Hausdorff; \(\pi\) continuous and open), together with~\ref{ax:norm}--\ref{ax:nbhd}, subsume Definition~13.4 of Fell and Doran~\cite{FellDoran1988Representations} (see also Kreidler and Siewert~\cite{kreidler_gelfand-type_2020} and Siewert~\cite{Siewert2020WeightedKoopman}), which requires only the locally compact Hausdorff axioms on the base. Condition~\ref{ax:nbhd} asserts that the topology of \(\EE\) near the zero section is determined by the fiber norms and the base topology; it can be viewed as a topological analog of local triviality~\cite[Remark~13.9]{FellDoran1988Representations}. We at no point rely on local triviality of the bundles, even though in applications local triviality occurs naturally.

A \emph{section} of \(\EE\) is a map \(\sigma \colon B \to \EE\) with \(\pi \circ \sigma = \id_B\); if it is continuous we call \( \sigma\) a continuous section.

\begin{remark}[Minimal topological hypotheses on the base]
\label{rem:minimal-base-topology}
The proofs in this paper invoke only three topological properties
of \(B\):
\begin{enumerate}[label=\textup{(\roman*)}]
  \item \emph{local compactness}, used to extract compact neighborhoods \(K \subseteq B\) of a basepoint on which fiber-uniform continuity arguments apply;
  \item \emph{Hausdorff separation} of \(B\), which guarantees uniqueness of limits, as well as being necessary for \cref{prop:local-sections}, which much of our argumentation relies upon;
  \item \emph{first countability} of \(B\), used to reduce continuity verifications on \(\EE\) to sequential arguments (\cref{lem:first-countable}, \cref{lem:uniform-limit}), and when proving the invariance of continuous bundle maps under our fixed point operator, for example, in \cref{sec:App}.
\end{enumerate}
The hypothesis ``\(B\) locally compact metric'' is a standard and convenient way to secure (i)--(iii) simultaneously: every concrete base appearing in this paper---smooth manifolds, Kawan's compact metrizable space of admissible control functions~\cite[Prop.~1.14]{Kawan2013InvarianceEntropy}---is metric.  Replacing ``metric'' by ``locally compact Hausdorff and first countable'' would change none of the arguments below; we retain ``metric'' for compatibility with the cited literature and brevity of exposition.
\end{remark}

We will need two standard facts about Banach bundles, both of which require the standing topological assumptions (in particular, local compactness of~\(B\)): a convergence criterion characterizing the topology of~\(\EE\) via local sections, and the existence of continuous local sections through every point of~\(\EE\).

\begin{lemma}[Convergence criterion]\label{lem:conv-criterion}
Let \((\EE, \pi, B)\) be a Banach bundle, \(w_0 \in \EE\), and \(\sigma \colon W \to \EE\) a continuous section on an open neighborhood \(W\) of \(b_0 := \pi(w_0)\) with \(\sigma(b_0) = w_0\). A net \((w_\lambda) \subseteq \EE\) converges to \(w_0\) in \(\EE\) if and only if
\begin{enumerate}[label=\textup{(\roman*)}]
  \item \(\pi(w_\lambda) \to b_0\) in~\(B\), and
  \item \(\norm{w_\lambda - \sigma(\pi(w_\lambda))} \to 0\) in \(\R\) \textup{(the difference being formed in the common fiber \(E_{\pi(w_\lambda)}\))}.
\end{enumerate}
\end{lemma}

\begin{proof}
This is given in most works involving Banach bundles, e.g.~\cite{Siewert2020WeightedKoopman,kreidler_gelfand-type_2020,FellDoran1988Representations}.
\end{proof}

\begin{proposition}[Existence of local sections through a point]\label{prop:local-sections}
Let \((\EE, \pi, B)\) be a Banach bundle over a locally compact Hausdorff base.  For every \(w_0 \in \EE\) there exist an open neighborhood \(W \subseteq B\) of \(\pi(w_0)\) and a continuous section \(\sigma \colon W \to \EE\) with \(\sigma(\pi(w_0)) = w_0\).
\end{proposition}

\begin{proof}
This is the Douady--dal Soglio-H\'erault section-existence theorem; see Fell and Doran~\cite[Appendix~C]{FellDoran1988Representations}.
\end{proof}

\begin{remark}\label{rem:paracompact-base}
  Since the Douady--dal Soglio-H\'erault section-existence theorem, \cite[Appendix~C]{FellDoran1988Representations}, is also true when \( B\) is \textit{paracompact}, we want to remark here that a similar treatment is potentially possible with this assumption instead, although not without immense restructuring, since much of our argumentation explicitly relies on the existence of compact neighborhoods.
\end{remark}

Two further consequences of the standing topological assumptions will be needed in \cref{sec:ED} and beyond: the total space is first countable, and fiber-norm-uniform limits of continuous sections are continuous.

\begin{lemma}[First countability of the total space]\label{lem:first-countable}
The total space \(\EE\) is first countable.  In particular, continuity of a map with domain \( \mathcal{E}\) at a point can be verified along sequences.
\end{lemma}

\begin{proof}
Let \(w_0 \in \EE\) and \(b_0 := \pi(w_0)\).  Since \(B\) is metric, it has a countable basis \((U_n)_{n \in \N}\) of open neighborhoods at~\(b_0\).  By \cref{prop:local-sections} choose a continuous section \(\sigma \colon W \to \EE\) on an open neighborhood \(W\) of~\(b_0\) with \(\sigma(b_0) = w_0\), and without loss of generality \(U_n \subseteq W\) for every~\(n\).  For \(n, k \in \N\), set
\begin{equation*}
  V_{n,k} := \bigl\{ v \in \pi^{-1}(U_n) :\norm{v - \sigma(\pi(v))} < 1/k \bigr\}.
\end{equation*}
Each \(V_{n,k}\) contains \(w_0\) (since \(\sigma(b_0) = w_0\)) and is open in~\(\EE\): the map \(v \mapsto v - \sigma(\pi(v))\) is continuous on \(\pi^{-1}(W)\) by~\ref{ax:add} and~\ref{ax:scal} and continuity of \(\sigma\) and \(\pi\), and the fiber norm is continuous by~\ref{ax:norm}, so \(V_{n,k}\) is the intersection of \(\pi^{-1}(U_n)\) with the preimage of \([0, 1/k)\) under a continuous map.  The converse implication in~\cref{lem:conv-criterion} shows that every open neighborhood of \(w_0\) in \(\EE\) contains some \(V_{n,k}\). Hence, \((V_{n,k})_{n,k \in \N}\) is a countable neighborhood basis of \(w_0\).

The final assertion is the standard fact that, in a first countable space, continuity at a point is equivalent to sequential continuity at that point.
\end{proof}

\begin{lemma}[Uniform limit of sections]\label{lem:uniform-limit}
Let \(W \subseteq B\) be open and \((\sigma_m)_{m \in \N}\) a sequence of continuous sections \(\sigma_m \colon W \to \EE\).  If \(\sigma \colon W \to \EE\) is a section satisfying
\begin{equation*}
  \sup_{b \in W}\norm{\sigma_m(b) - \sigma(b)}
  \;\xrightarrow{m \to \infty}\; 0,
\end{equation*}
then \(\sigma\) is continuous on \(W\).
\end{lemma}

\begin{proof}
Fix \(b_0 \in W\); we show \(\sigma\) is continuous at \(b_0\).  Since \(W\) is first countable at \(b_0\) (inherited from the metric topology on~\(B\)) and \(\EE\) is Hausdorff, it suffices to show \(\sigma(b_n) \to \sigma(b_0)\) in \(\EE\) for every sequence \((b_n) \subseteq W\) with \(b_n \to b_0\).

By \cref{prop:local-sections}, choose a continuous local section \(\tau \colon W_\tau \to \EE\) with \(\tau(b_0) = \sigma(b_0)\); for large \(n\), \(b_n \in W_\tau \cap W\). By \cref{lem:conv-criterion} applied to \(\tau\) at \(\sigma(b_0)\), it suffices to show \(\norm{\sigma(b_n) - \tau(b_n)} \to 0\).  For any \(m \in \N\),
\begin{align*}
  \norm{\sigma(b_n) - \tau(b_n)}
  &\;\leq\; \norm{\sigma(b_n) - \sigma_m(b_n)}
          + \norm{\sigma_m(b_n) - \tau(b_n)} \\
  &\;\leq\; \varepsilon_m
          + \norm{\sigma_m(b_n) - \tau(b_n)},
\end{align*}
where \(\varepsilon_m := \sup_{b \in W}\norm{\sigma_m(b) - \sigma(b)}\).  The map \(b \mapsto \norm{\sigma_m(b) - \tau(b)}\) is continuous on \(W \cap W_\tau\) by~\ref{ax:norm}--\ref{ax:scal}, so \(\norm{\sigma_m(b_n) - \tau(b_n)} \to \norm{\sigma_m(b_0) - \tau(b_0)} = \norm{\sigma_m(b_0) - \sigma(b_0)} \leq \varepsilon_m\). Therefore, \(\limsup_n \norm{\sigma(b_n) - \tau(b_n)} \leq 2\varepsilon_m\) for every \(m\); letting \(m \to \infty\) gives \(\norm{\sigma(b_n) - \tau(b_n)} \to 0\).
\end{proof}

Taken together, \cref{lem:conv-criterion,prop:local-sections,lem:first-countable,lem:uniform-limit} allow us to verify continuity of bundle-valued maps along local sections: to show that a map \(\Theta \colon \EE \to \EE\) covering \(\id_B\) is continuous at~\(v_0\), it suffices to exhibit a continuous local section~\(\sigma\) with \(\sigma(\pi(v_0)) = v_0\), check that \(b \mapsto \Theta(\sigma(b))\) is a continuous section, and check that \(\norm{\Theta(v_n) - \Theta(\sigma(\pi(v_n)))} \to 0\) whenever \(v_n \to v_0\).

\subsection{Skew-product flows}\label{subsec:SPF}

Skew-product flows have been in active use since at least the 1970s, finding heavy use in the work of Sacker and Sell~\cite{SackerSell1974ExistenceDichotomies,SackerSell1976ExistenceDichotomies,SackerSell1978SpectralTheory}, Chow and Yi~\cite{chow_center_1994}, and Chow and Leiva~\cite{chow_existence_1995}, and have since become an established notion of modeling non-autonomous dynamics and a focus of study in and of themselves~\cite{kloeden_nonautonomous_2011,ChiconeLatushkin1999EvolutionSemigroups} and have recently even found application in control theory~\cite{ColoniusFabbri2025NonautonomousControl}. 

Let \(\theta \colon \R \times B \to B\) be a continuous flow on the base.  We write \(b \cdot t := \theta_t(b)\).

\begin{definition}[Skew-product flow]\label{def:SPF}
A \emph{skew-product flow} (SPF) on \(\EE\) over~\(\theta\) is a continuous map \(\Phi \colon \R \times \EE \to \EE\) satisfying:
\begin{enumerate}[label=\textup{(\roman*)}]
  \item \(\pi \circ \Phi_t = \theta_t \circ \pi\)
    \quad (fiber-respecting);
  \item \(\Phi_0 = \id_\EE\);
  \item \(\Phi_{t+s} = \Phi_t \circ \Phi_s\)
    \quad (flow property).
\end{enumerate}
We write \(\Phi_t := \Phi(t, \cdot)\).  The commutative diagram
\begin{equation}\label{eq:SPF}
\begin{tikzcd}
  \EE \ar[r, "\Phi_t"] \ar[d, "\pi"'] & \EE \ar[d, "\pi"] \\
  B \ar[r, "\theta_t"'] & B
\end{tikzcd}
\end{equation}
says that \(\Phi\) maps fibers to fibers, covering~\(\theta\).
\end{definition}

\begin{definition}[Linear SPF]\label{def:linearSPF}
An SPF \(\Phi\) is \emph{linear} if each restriction \(\Phi(t,b) := \Phi_t|_{E_b} \colon E_b \to E_{b \cdot t}\) is a bounded linear operator.  The flow property then takes the form of a \emph{composition rule}:
\begin{equation}\label{eq:cocycle-def}
  \Phi(t+s, b) = \Phi(t, b \cdot s)\, \Phi(s, b)
  \qquad (t, s \in \R,\; b \in B).
\end{equation}
Since \(\Phi\) is two-sided, each \(\Phi(t, b)\) is invertible with  \[
\Phi(t, b)^{-1} = \Phi(-t, b \cdot t).
\]
\end{definition}

On a non-trivial bundle the operators \(\Phi(t,b) \in \LL(E_b, E_{b \cdot t})\) map between \emph{a priori different} Banach spaces; thus one has to be careful in type-checking. We will, for the sake of brevity, not spell this out every single time. Another consequence of this is that the usual notion of strong continuity does not make sense anymore in our setting. We substitute this with a weaker, but related notion. 

\begin{definition}[Continuity along sections]
\label{def:CAS}
Let \(h \colon B \to B\) be a continuous map and let \(T = (T(b))_{b \in B}\) be a family of bounded operators \(T(b) \in \LL(E_b, E_{h(b)})\) covering~\(h\).  We say that \(T\) is \emph{continuous along sections} if, for every continuous section \( \sigma\), the map \(b \mapsto T(b)\,\sigma(b)\) is a continuous section of~\(\EE\) along~\(h\), i.e.\ a continuous map satisfying \[
    \pi \circ T\sigma = h
.\] 
\end{definition}

On a trivial bundle with a single fiber, choosing \(\sigma\) constant reduces \cref{def:CAS} to strong operator continuity of \(b \mapsto T(b)\); on a general bundle it is the genuine bundle-theoretic analog.

\subsection{Morphisms and conjugacies}\label{subsec:conjugacies}

Let \(\Phi\) be an SPF over \(\theta\) on \((\EE, \pi, B)\) and \(\Phi'\) an SPF over \(\theta'\) on \((\EE', \pi', B')\).

\begin{definition}[Morphism and conjugacy]\label{def:morphism}
A \emph{morphism} from \(\Phi\) to \(\Phi'\) is a pair \((h, H)\) where \(h \colon B \to B'\) intertwines the base flows (\(h \circ \theta_t = \theta'_t \circ h\)) and \(H \colon \EE \to \EE'\) is a bundle map over~\(h\) (\(\pi' \circ H = h \circ \pi\)) satisfying \(H \circ \Phi_t = \Phi'_t \circ H\) for all \(t\). If \(h\) is a homeomorphism and each \(H|_{E_b}\) is a homeomorphism and the inverse \( H^{-1}\) is continuous, then \((h, H)\) is an \emph{equivalence}. If, in addition, \( \mathcal{E} = \mathcal{E}'\), \(\theta = \theta'\) and \( h = \id\), then the equivalence is called a \emph{conjugacy}.
\end{definition}

\begin{remark}\label{rem:readability}
  We want to remark at this point that all the statements we are going to derive here will, for readability's sake, focus on the case of conjugacies, so when \( h = \id\) on \( B\). This does not mean that the case where \( h\) is nontrivial is any more difficult. In fact all the proofs will translate mutatis mutandis, by replacing \( b\) with \( h(b)\) if necessary.
\end{remark}

A morphism can be expressed by the commutativity of
\[\begin{tikzcd}[ampersand replacement=\&,sep=small]
	{\R \times \mathcal{E}} \&\&\&\& {\mathcal{E}} \&\& \\
	\\
	\&\& {\R \times \mathcal{E}'} \&\&\&\& {\mathcal{E}'} \\
	\\
	{\R \times B} \&\&\&\& B \\
	\\
	\&\& {\R \times B'} \&\&\&\& {B'}
  \arrow["\Phi", from=1-1, to=1-5]
	\arrow["{\id \times H}", from=1-1, to=3-3]
	\arrow[from=1-1, to=5-1]
	\arrow["H"', from=1-5, to=3-7]
	\arrow[from=1-5, to=5-5]
	\arrow[from=3-7, to=7-7]
	\arrow["\theta"{pos=0.6}, from=5-1, to=5-5]
	\arrow["{\id \times h}", from=5-1, to=7-3]
	\arrow["h", from=5-5, to=7-7]
	\arrow["{\theta'}", from=7-3, to=7-7]
	\arrow["{\Phi'}"{pos=0.4}, from=3-3, to=3-7,crossing over]
  \arrow[from=3-3, to=7-3, crossing over]
\end{tikzcd}\]

\subsection{Classical equivalences as special cases}
\label{subsec:examples}

The bundle conjugacy framework unifies several classical equivalence concepts.

\begin{proposition}\label{prop:classical}
The following are special cases of \cref{def:morphism}:
\begin{enumerate}[label=\textup{(\alph*)}]
  \item \emph{Autonomous conjugacy:} \(B = \{*\}\),
    \(\EE = X\), \(\Phi = \psi\) (a flow on~\(X\)).
    A fiber conjugacy is a homeomorphism
    \(H \colon X \to X\) with \(H \circ \psi_t = \psi'_t \circ H\).
  \item \emph{Nonautonomous conjugacy:} \(B = \R\),
    \(\EE = \R \times \R^N\) (trivial), \(\theta_t(s) = s + t\).
    A fiber conjugacy over \(\id_\R\) is a family
    \[H_s \colon \R^N \to \R^N\] with
    $H_{s+t} \circ \psi(s{+}t, s, \cdot) =
    \psi'(s{+}t, s, \cdot) \circ H_s$.
  \item \emph{Equivalence with time shift:} as in~\textup{(b)} but
    with a base map \(h \neq \id_\R\); since \(h\) intertwines the
    translations, \(h(s) = s + c\) for some \(c \in \R\).
  \item \emph{Control systems:} the linearization along a
    hyperbolic equilibrium section
    \(\sigma \colon \mathcal{U}  \to M\)
    of a control-affine system on a manifold~\(M\)~\cite{ColoniusFabbri2025NonautonomousControl,ColoniusKliemann2000DynamicsControl}
    yields a linear SPF on the tangent bundle along~\(\sigma\),
    which is in general non-trivial.
\end{enumerate}
\end{proposition}
\section{Exponential dichotomies of skew-product flows}  \label{sec:ED}

Although thoroughly studied and well understood in the community for evolution families, the more general case of exponential dichotomies for skew-product flows has been of more niche interest~\cite{chow_center_1994,LatushkinEtAl1996EvolutionarySemigroups,BlumenthalLatushkin2019SelgradeDecomposition,Rau1996HyperbolicLinear} and has mainly remained in the realm of trivial bundles. A notion of exponential dichotomy for skew-products on non-trivial bundles has so far been missing from the literature as far as we can tell, even though the interest in properly non-trivial bundles~\cite{gierz_bundles_1982} and their dynamics has been developed in works like~\cite{kreidler_gelfand-type_2020} and in the thesis~\cite{Siewert2020WeightedKoopman}. In this section, we will give notions of uniform and nonuniform exponential dichotomies and explore how the richness of the class of Banach bundles enables us to treat both cases simultaneously. This is exemplified in \cref{thm:ED-NUED-equiv} and \cref{thm:StNUED-StED-Equiv} and their subsequent application in later sections. 

\subsection{Definitions}\label{subsec:ED-def}

\begin{definition}[Invariant projection family]\label{def:projfamily}
Let \(\Phi\) be a linear SPF on \((\EE, \pi, B)\).  An \emph{invariant projection family} is a collection \(P = (P(b))_{b \in B}\) of bounded projections \(P(b) \in \LL(E_b)\) such that:
\begin{enumerate}[label=\textup{(\roman*)}]
  \item \(P\) is continuous along sections (\cref{def:CAS});
  \item \( P\) is invariant under \( \Phi\), i.e.  \[
  P(b \cdot t)\,\Phi(t,b) = \Phi(t,b)\,P(b)\] 
  for all \(t \in \R, b \in B\).
\end{enumerate}

We set \(Q(b) := \id_{E_b} - P(b)\),
\(E^s_b := \mathrm{im}\, P(b)\),
\(E^u_b := \ker P(b)\).
\end{definition}

\begin{remark}\label{rem:P-continuity-local}
  We want to note here that it is possible to only require continuity along local sections, because the base space, a locally compact metric space, admits cutoff functions. This leads to the equivalence of both definitions in our setting.
\end{remark}

From here on out we will always assume that \( P\) is an invariant projection in the above sense.

\begin{definition}[Exponential dichotomy]\label{def:ED}
A linear SPF \(\Phi\) admits an \emph{exponential dichotomy (ED) with data \((P, D, \delta)\)} if there exist an invariant projection family~\(P\), \(D \geq 1\), and \(\delta > 0\) with
\begin{align}
  \label{eq:ED-stable}
  \norm{\Phi_t\, P\, v}
  &\leq D\, e^{-\delta t}\, \norm{v}
  &&(t \geq 0), \tag{EDs}\\[4pt]
  \label{eq:ED-unstable}
  \norm{\Phi_{-t}\, Q\, v}
  &\leq D\, e^{-\delta t}\, \norm{v}
  && \tag{EDu}(t \geq 0),
\end{align}
for all \(v \in \EE\).
\end{definition}

\begin{remark}\label{rem:D-meaning}
  The constant \(D\) measures how far the fiber norms are from being adapted to the dichotomy; \cref{thm:ED-NUED-equiv} below shows that \(D\) can always be reduced to~\(1\) by a change of bundle structure, at the cost of a slightly smaller rate.
\end{remark}

\begin{remark}[Projector bounds]\label{rem:PQ-bounds}
Setting \(t = 0\) in~\eqref{eq:ED-stable}--\eqref{eq:ED-unstable} gives \(\norm{P(b)v} \leq D\norm{v}\) and \(\norm{Q(b)v} \leq D\norm{v}\) for all \(v \in E_b\), hence \[
\sup_b \norm{P(b)} \leq D \qquad \sup_b \norm{Q(b)} \leq D.\] 
These uniform bounds are used silently in \cref{sec:HGT}.
\end{remark}

\subsection{Nonuniform dichotomies}\label{subsec:ED-equiv}
We now proceed to weaken the concept of an exponential dichotomy in a way which has its genesis in the landmark works~\cite{oseledec_lyapunov_1968,pesin_families_1976} and is in general a more typical notion of dichotomy, where one allows exponential spoiling along trajectories to take place. To generalize the concept we need to introduce functions that model this spoiling along curves in the base space. 

\begin{definition}[\( \varepsilon\)-slowly varying functions]\label{def:SlowVar}
  We say \( N : B \to [1,\infty)\) is \( \varepsilon\)-slowly varying if it is continuous and for all \( t \in \R\) and \( b \in B\) we have \[
      N(b \cdot t ) \leq e^{\varepsilon |t| }N(b)
  .\] 
\end{definition}

Having defined these functions, we use them to define both nonuniform exponential dichotomy and equivalent bundle structures, which, as we will prove, connect and unify the concepts.

\begin{definition}[Nonuniform exponential dichotomy]\label{def:NUED}
  A linear SPF \( \Phi\) admits a \emph{nonuniform exponential dichotomy} (NUED) with data \[
      (P,D,\varepsilon,\delta)
  \] if there exists an invariant projection \( P\), constants \( D \geq 1\) and \( \delta >0\) and an \( \varepsilon\)-slowly varying function \( N : B \to [1,\infty)\) such that     
  \begin{align}
      \label{eq:NUED-stable}
      \norm{\Phi_t\,P\,v}
      &\leq D\, N(b)\, e^{-\delta t}\, \norm{v}, \tag{NUEDs}\\[4pt]
      \label{eq:NUED-unstable}
      \norm{\Phi_{-t}\,Q\,v}
      &\leq D\, N(b)\, e^{-\delta t}\, \norm{v}, \tag{NUEDu}
    \end{align}
    for all \(t \geq 0\), \(b \in B\), \(v \in E_b\).
\end{definition}

The term ``NUED'' is justified by viewing the nonuniformity not
as a separate notion of hyperbolicity, but as the manifestation
of a genuine ED when the fiber norms are not adapted to the
dynamics, which will become clear in the following pages.

\begin{definition}[\(\varepsilon\)-equivalent bundle structures]
\label{def:eps-equiv}
Two Banach bundle structures \(\norm{\cdot}_b\) and \(\norm{\cdot}'_b\) on \((\EE, \pi, B)\) are \emph{\(\varepsilon\)-equivalent} if there exists an \( \varepsilon\)-slowly varying function \(N \colon B \to [1, \infty)\) with
\begin{equation}
  \norm{v}_b \leq \norm{v}'_b \leq N(b)\,\norm{v}_b
\end{equation}
for all \(v \in E_b\).
\end{definition}

\begin{remark}\label{rem:eps-equiv-topology}
  Generally speaking an \( \varepsilon\)-equivalent Banach bundle structure induces the same topology, since locally they will differ only by a finite factor. However, it is not true that this Banach bundle is the same in a strict sense, because the identity map \( \id_\mathcal{E} : (\mathcal{E}, \|\cdot \| ) \to (\mathcal{E}, \|\cdot \|' )\) is not a bundle morphism, even though it is a homeomorphism of the total space, since the distortion function will in general not be bounded. Barreira and Pesin~\cite{BarreiraPesin2007NonuniformHyperbolicitya} call such functions \( \varepsilon\)-tempered, whereas analogous tempering appears in the theory of random dynamical systems, where the Lyapunov norm supplied by the multiplicative ergodic theorem satisfies \( B_{\varepsilon}(\theta_t \omega) \leq e^{\varepsilon |t| }B_{\varepsilon}(\omega)\), which allows nonuniformity to enter the Hartman--Grobman theorem of~\cite{Coayla-TeranEtAl2007HartmanGrobman}.
\end{remark}

\begin{theorem}[NUED and adapted exponential dichotomies]
\label{thm:ED-NUED-equiv}
Let \(\Phi\) be a linear SPF on \((\mathcal{E},\pi,B)\), and let \(\varepsilon \geq 0\).

\begin{enumerate}[label=\textup{(\roman*)}]
\item If \(\Phi\) admits an NUED with constants \((P,D,\varepsilon,\delta)\), then for every \(\delta'\in (0,\delta)\) there exists an \(\varepsilon\)-equivalent Banach bundle norm \(\|\cdot\|'\) such that \(\Phi\), viewed with respect to \(\|\cdot\|'\), has an ED with data \((P,1,\delta')\).

\item Conversely, if there exists an \(\varepsilon\)-equivalent Banach bundle norm \(\|\cdot\|'\) such that \(\Phi\) has an ED with respect to \(\|\cdot\|'\) with data \((P,D,\delta')\), then \(\Phi\), viewed with respect to the original norm, admits an NUED with constants \((P,D,\varepsilon,\delta')\).
\end{enumerate}
Thus, nonuniform exponential dichotomies are precisely uniform exponential dichotomies after passing to an \(\varepsilon\)-equivalent Banach bundle norm.
\end{theorem}

\begin{proof}
The proof is fundamentally based on the construction of Lyapunov norms which are adapted to the dynamics and can be viewed as the bundle analog of the construction Barreira, Dragi\v{c}evi\'c, and Valls use in the proof of~\cite[Proposition~5.6]{BarreiraEtAl2018AdmissibilityHyperbolicity}.
Assume first that \(\Phi\) admits an \(\varepsilon\)-NUED with constants \((P,D,\varepsilon,\delta)\), and let \(N\) be a witnessing \(\varepsilon\)-slowly varying function. Fix \(\delta'\in(0,\delta)\). Define
\begin{align*}
  \|v\|^s_b &:= \sup_{t \geq 0} e^{\delta't}\|\Phi(t,b)P(b)v\|_{b\cdot t}, \\
  \|v\|^u_b &:= \sup_{t\ge0} e^{\delta't} \|\Phi(-t,b)Q(b)v\|_{b\cdot(-t)},
\end{align*}
and \[
  \|v\|'_b:=\|v\|^s_b+\|v\|^u_b.
\]

At \(t=0\), \[
  \|v\|_b \leq \|P(b)v\|_b+\|Q(b)v\|_b \leq \|v\|'_b.
\]
The NUED estimates give \begin{equation}\label{eq:Lyap1}
  \|v\|^s_b,\ \|v\|^u_b \leq D N(b)\|v\|_b,
\end{equation}
hence \[
  \|v\|_b \leq \|v\|'_b \leq 2D N(b)\|v\|_b.
\]
Since \(2DN\) is again \(\varepsilon\)-slowly varying, the new norm is \(\varepsilon\)-equivalent to the old one.

It remains to check that \(\|\cdot\|'\) is a Banach bundle norm, in particular the norm is continuously varying over \( B\). The argument is the same for the stable and unstable seminorms, so we only write the stable one. Let \(\tau:W_0 \to \mathcal{E}\) be a continuous local section. 
Let \( b_0 \in B\) and choose a compact neighborhood of \( b_0 \), say \( K \subset W_0\), which exists by our assumptions. Define \[
  \varphi(b,t) := e^{\delta't} \|\Phi(t,b)P(b)\tau(b)\|_{b\cdot t}
,\] which is continuous on \(K \times [0,\infty)\) as the composition of continuous functions, since \( b \mapsto P(b)\tau(b)\) is continuous by definition. Hence, for each \( T > 0\) the map \[
  S_T(b) =\sup_{0 \leq t \leq T}\varphi(b,t)
\]
is continuous on \(K\). Moreover, \[
  \varphi(b,t) \leq D N(b)e^{-(\delta-\delta')t}\|\tau(b)\|_b,
\] and since \( N\) and \( \|\tau(\cdot)\| \) are continuous we have \[
    C_K := \sup_{b \in K} D N(b) \|\tau(b)\| < \infty   
.\] This means now that \[
    \sup_{b \in K, t \geq T} \varphi(b,t)\leq C_K e^{-(\delta-\delta')T} \to 0
\] as \( T \to \infty \). By this estimate we see immediately \[
  | \|\tau(b)\|^{s}_b  - S_T(b)| \leq C_K e^{-(\delta-\delta')T}  \to 0
.\] Hence, \(b\mapsto\|\tau(b)\|^s_b\) is continuous as the uniform limit of continuous functions. To see the continuity of \( \|\cdot \|^{s} \) we invoke \cref{lem:conv-criterion}. Choose \( v_n \to v_0\) with \( b_n := \pi(v_n) \to b_0 = \pi(v_0)\) and a continuous local section \( \tau : W_1 \subset K \to \mathcal{E} \) such that \[
    \tau(b_0)=v_0
\] and up to finitely many terms \( b_n \in W_1\). By \cref{lem:conv-criterion} we have \[
    \|v_n -\tau(b_n)\| \to 0 
.\] Put \( h_n = v_n -\tau(b_n)\). Since \( b_n \to b_0\) the values of \( N(b_n)\) are bounded, hence, by \cref{eq:Lyap1}, we have \[
    \|h_n\|^{s} \leq D N(b_n) \|h_n\| \to 0 
.\] Hence, the inverse triangle inequality gives now \[
  \left| \|v_n\|^{s} - \|\tau(b_n)\|^{s} \right| \leq \|v_n - \tau(b_n)\|^{s} = \|h_n\|^{s} \to 0   
.\] Hence, we have \[
    \|v_n\|^{s} = \|\tau(b_n)\|^{s} + \left( \|v_n\|^{s} - \|\tau(b_n)\|^{s}\right)  
,\] where the first summand is, by continuity of \( \tau\) with respect to \( \|\cdot \|^{s} \), approaching \( \|v_0\|^{s} \) and the second term goes to \( 0\). Thus, we have \[
    \|v_n\|^{s} \to \|v_0\|^{s}
,\] proving the continuity of \( \|\cdot \|^{s} \). The unstable estimate follows thereon analogously. 

The remaining Banach bundle axioms are unchanged: addition and scalar multiplication are the original maps; the zero-neighborhood axiom follows from \(\|v\|_b\le\|v\|'_b\); and completeness of each fiber follows from equivalence of the two norms on that fiber.

Finally, for \(t\geq 0\), \[
  \|\Phi(t,b)P(b)v\|'_{b\cdot t} = \|\Phi(t,b)P(b)v\|^s_{b\cdot t} \leq e^{-\delta't}\|v\|^s_b \leq e^{-\delta't}\|v\|'_b.
\]
The unstable estimate is analogous. Thus, \(\Phi\) has an ED with data \((P,1,\delta')\) in the Lyapunov norm.

Conversely, suppose \(\|\cdot\|'\) is \(\varepsilon\)-equivalent to \(\|\cdot\|\), witnessed by \(N\), and that \(\Phi\) has ED data \((P,D,\delta')\) in \(\|\cdot\|'\). Then, for \(t \geq 0\), \[
  \|\Phi(t,b)P(b)v\|_{b\cdot t} \leq \|\Phi(t,b)P(b)v\|'_{b\cdot t} \leq D e^{-\delta't}\|v\|'_b \leq D N(b)e^{-\delta't}\|v\|_b.
\]
The unstable estimate is the same. Hence, \(\Phi\), viewed with respect to the original norm, admits an NUED with constants \((P,D,\varepsilon,\delta')\).
\end{proof}

This is the Banach-bundle analog of~\cite[Proposition~5.6]{BarreiraEtAl2018AdmissibilityHyperbolicity}; the function \(N\) measures this distortion. For the trivial bundle \(\EE = \R \times X\) with \(N(s) = e^{\varepsilon|s|}\), \cref{eq:NUED-stable} reduces to the classical definition, given for example in~\cite{BarreiraValls2011SimpleProof}.

\begin{remark}
  We do want to remark here that this result has been fundamental for our motivation to develop this theory in the case of non-trivial bundles. This is because this construction, compare with~\cite[Proposition~5.6]{BarreiraEtAl2018AdmissibilityHyperbolicity}, turns trivial bundles \(\R \times X\) into something that can no longer be expressed as a Cartesian product with a single Banach space, but into something generally non-trivial as a Banach bundle. Thus, it makes sense to think of this construction as a fundamentally bundle-theoretic one. We will see the lack of triviality will cause no problems.
\end{remark}

Remarkably, \cref{thm:ED-NUED-equiv} holds for every \(\varepsilon > 0\) without restriction on \(\varepsilon / \delta\). For the linearization theorem (\cref{thm:HG}), the relevant quantity is the dichotomy data in the adapted bundle; starting from NUED, the Lyapunov norms yield \(D = 1\) and rate \(\delta' < \delta\).

\subsection{Strong dichotomies}

Strong dichotomies are the correct strengthening needed to prove higher regularity of the conjugacy. We here choose to introduce the concept in its broadest sense, where we allow different rates for forwards and backwards time on the stable and unstable subbundles respectively. This will later on contribute to the wordiness of the exposition, but will provide the possibility for the most detailed analysis of the dichotomy.

\begin{definition}[Strong ED]\label{def:strongED}
A linear SPF \(\Phi\) has a \emph{strong exponential dichotomy with data \((P,D, \delta_s, \bar\delta_s, \delta_u, \bar\delta_u)\)} if there exist an invariant projection family~\(P\) and constants \(D \geq 1\) and \(0 < \delta_s \leq \bar\delta_s\) and \(0 < \delta_u \leq \bar\delta_u\), such that the \emph{upper bounds with separate rates}
\begin{align}
  \label{eq:strongED-upper-s}
  \norm{\Phi_t\, Pv} &\leq D\, e^{-\delta_s\, t}\, \norm{v} &&(t \geq 0), \tag{strEDup-s}\\ 
  \label{eq:strongED-upper-u}
  \norm{\Phi_{-t}\, Qv} &\leq D\, e^{-\delta_u\, t}\, \norm{v} &&(t \geq 0), \tag{strEDup-u}
\end{align}
and the following \emph{lower bounds} hold:
\begin{align}
  \label{eq:strongED-s}
  \norm{\Phi_t\, Pv} &\geq D^{-1}\, e^{-\bar\delta_s\, t}\,\norm{Pv}  &&(t \geq 0), \tag{strEDlow-s}\\
  \label{eq:strongED-u}
  \norm{\Phi_{-t}\, Qv} &\geq D^{-1}\, e^{-\bar\delta_u\, t}\,\norm{Qv} &&(t \geq 0), \tag{strEDlow-u}
\end{align}
for all \(v \in \EE\).   
In particular, taking \(\delta := \min\{\delta_s, \delta_u\}\) in~\eqref{eq:strongED-upper-s}--\eqref{eq:strongED-upper-u} shows that \(\Phi\) has an ordinary ED (\cref{def:ED}) with data \((P, D, \delta)\). The \emph{spectral-gap ratio} is
\begin{equation}\label{eq:alpha0}
  \alpha_0 := \min\!\left\{ \frac{\delta_s}{\bar\delta_s},\;\frac{\delta_u}{\bar\delta_u} \right\} \in (0, 1].
\end{equation}
\end{definition}

It might seem unintuitive that we choose to define the lower bounds by comparing the propagated norm with the norms of \( Pv\) and \( Qv\) and not to the norm of \( v\). We justify this by arguing that the inconsistency is justified by the convenience that the estimates \cref{eq:strongED-upper-s} and \cref{eq:strongED-upper-u} give us uniform bounds on the norm of the projection for \( t= 0\). 

The strong dichotomy estimates we will assume are not just measures of convenience; they fundamentally allow us to control the behavior of the fixed-point operator defined later more finely. This is due to them preventing superexponential growth and decay. As our next lemma will show, assuming strong dichotomy has very particular growth rate implications necessary for proving higher regularity of the conjugation we will prove exists in \cref{sec:HGT}.

\begin{lemma}\label{lem:growth-SED}
  Assume that \(\Phi\) admits a strong exponential dichotomy with data
\((P,D,\delta_s,\bar\delta_s,\delta_u,\bar\delta_u)\). Then for \(t \geq 0\) and \(z\in \mathcal{E}\), \[
\|\Phi_{-t}Pz\| \leq D^2 e^{\bar\delta_s t}\|z\|, \qquad \|\Phi_t Qz\| \leq D^2 e^{\bar\delta_u t}\|z\|.
\]
Consequently, for every \(\alpha\in(0,1)\), \[
\|\Phi_{-t}z\|^\alpha \leq (D^2)^\alpha e^{\alpha\bar\delta_s t}\|Pz\|^\alpha + D^\alpha e^{-\alpha\delta_u t}\|Qz\|^\alpha,
\]
and symmetrically \[
\|\Phi_t z\|^\alpha \leq D^\alpha e^{-\alpha\delta_s t}\|Pz\|^\alpha + (D^2)^\alpha e^{\alpha\bar\delta_u t}\|Qz\|^\alpha.
\]
\end{lemma}

\begin{proof}
We prove the backward estimate on the stable component; the forward
estimate on the unstable component is identical, with
\(P,\delta_s,\bar\delta_s\) replaced by \(Q,\delta_u,\bar\delta_u\).

Fix \(z\in \mathcal{E}\) and \(t \geq 0\). Now calculate, using invariance of \( P\), \[
   \|Pz\|  =\|\Phi_{t}\Phi_{-t}Pz\| \geq D^{-1}e^{- \overline{\delta}_s t} \|\Phi_{-t}Pz\|.  
\] 
Rearranging, \[
  \|\Phi_{-t}Pz\| \leq D e^{ \overline{\delta}_s t} \|Pz\| \leq D^{2}e^{ \overline{\delta}_s t} \|z\|. 
\]

Now decompose \(z=Pz+Qz\). Then \[
  \|\Phi_{-t}z\| \leq \|\Phi_{-t}Pz\|+\|\Phi_{-t}Qz\|.
\]
Since \(0<\alpha<1\), the function \(r\mapsto r^\alpha\) is
subadditive: \[
  (a+b)^\alpha\le a^\alpha+b^\alpha.
\]
Therefore, \[
  \|\Phi_{-t}z\|^\alpha \leq \|\Phi_{-t}Pz\|^\alpha+ \|\Phi_{-t}Qz\|^\alpha,
\]
and inserting~\eqref{eq:strongED-upper-u} gives \[
  \|\Phi_{-t}z\|^\alpha \leq (D^2)^\alpha e^{\alpha\bar\delta_st}\|Pz\|^\alpha + D^\alpha e^{-\alpha\delta_ut}\|Qz\|^\alpha
.\]
The forward-time estimate follows in the same way, using \(z=Pz+Qz\), the forward stable ED estimate~\eqref{eq:strongED-upper-s}, and the forward unstable growth estimate.
\end{proof}

For \emph{autonomous} systems (\(\dot x = Ax\)), the eigenvalues
simultaneously determine both rates, so the strong ED is
automatic.  The distinction becomes non-trivial only in the
nonautonomous setting, where super-exponential growth and decay
can occur.

The spectral-gap ratio \(\alpha_0\) should be contrasted with the
Sternberg non-resonance condition~\cite{Sternberg1958StructureLocal} for
smooth linearization: H\"older regularity is controlled by
quantitative backward growth, while smooth regularity requires
algebraic solvability of homological equations.  The two
conditions are complementary.

\subsection{Strong nonuniform dichotomies}
\label{subsec:NUED-Hoelder}

Strong NUED can be converted into strong ED after passing to suitably chosen two-sided Lyapunov norms. The ordinary Lyapunov norms of \cref{thm:ED-NUED-equiv} are insufficient, because they adapt only the contracting directions and leave a nonuniform factor in the reverse-growth estimates. But this can be remedied! In fact, we can prove an analogous result to \cref{thm:ED-NUED-equiv} in which a strong NUED is turned into a strong ED in an adapted Banach bundle structure. 

\begin{definition}[Strong NUED]
A linear SPF \(\Phi\) has a strong nonuniform exponential dichotomy with data \[
(P,D,\varepsilon,\delta_s,\bar\delta_s,\delta_u,\bar\delta_u)
\] if there exist an invariant projection \( P\), an \( \varepsilon\)-slowly varying function \(N\) and constants \( D \geq 1\) and \(0 < \delta_s \leq \bar\delta_s\) and \(0 < \delta_u \leq \bar\delta_u\) such that
\begin{align}
\label{eq:strongNUED-upper-s}\|\Phi_tPv\| &\leq D N(b)e^{-\delta_s t}\|v\|, &&(t \geq 0)  \tag{strNUEDup-s}\\ 
\label{eq:strongNUED-upper-u}\|\Phi_{-t}Qv\| &\leq D N(b)e^{-\delta_u t}\|v\|, &&(t \geq 0)  \tag{strNUEDup-u}
\end{align}
and
\begin{align}
\label{eq:strongNUED-lower-s}\|\Phi_tPv\| &\geq D^{-1}N(b)^{-1}e^{-\bar\delta_s t}\|Pv\|, &&(t \geq 0) \tag{strNUEDlow-s}\\
\label{eq:strongNUED-lower-u}\|\Phi_{-t}Qv\| &\geq D^{-1}N(b)^{-1}e^{-\bar\delta_u t}\|Qv\|, &&(t \geq 0)   \tag{strNUEDlow-u}
\end{align}
for all \(v \in \mathcal{E}\).
\end{definition}

This definition, while new, provides a natural extension of the notion found 
in the literature~\cite{BarreiraEtAl2018AdmissibilityHyperbolicity,BarreiraValls2011SimpleProof}. 

We are now able to introduce the strong dichotomy analog of \cref{thm:ED-NUED-equiv}.

\begin{theorem}[Strong Lyapunov norms]\label{thm:StNUED-StED-Equiv}
Assume that \(\Phi\) has a strong nonuniform exponential dichotomy with data \[
(P,D,\varepsilon,\delta_s,\bar\delta_s,\delta_u,\bar\delta_u)
.\]
Then, for all rates with \[
0<\delta_s'<\delta_s, \qquad 0<\delta_u'<\delta_u,
\qquad \bar\delta_s'>\bar\delta_s+\varepsilon, \qquad \bar\delta_u'>\bar\delta_u+\varepsilon,
\]
there exists a \( 2\varepsilon\)-equivalent Banach bundle norm \(\|\cdot\|_b^{\mathrm{str}}\)
on the same underlying bundle, satisfying the estimate \[
  \|v\|_b \leq \|v\|^{\mathrm{str}}_b \leq D N(b)\bigl(\|P(b)v\|+\|Q(b)v\|\bigr) \leq 2D^2N(b)^2\|v\|_b ,
\] such that \(\Phi\) has a strong exponential
dichotomy with uniform data
\[
(P,1,\delta_s',\bar\delta_s',\delta_u',\bar\delta_u')
\]
with respect to \(\|\cdot\|^{\mathrm{str}}\).
\end{theorem}

\begin{proof}
For \(x\in E_b^s\), define \[
\|x\|_{s,b}^{\mathrm{str}} := \sup_{t\in\mathbb R} \rho_s(t)\|\Phi(t,b)x\|_{b\cdot t},
\]
where \[
\rho_s(t) = \begin{cases}
e^{\delta_s't}, & t\ge0,\\
e^{\bar\delta_s't}, & t\le0.
\end{cases}
\]
For \(y\in E_b^u\), define \[
\|y\|_{u,b}^{\mathrm{str}} := \sup_{t\in\mathbb R} \rho_u(t)\|\Phi(t,b)y\|_{b\cdot t},
\]
where \[
\rho_u(t) =
\begin{cases}
e^{-\bar\delta_u't}, & t\ge0,\\
e^{-\delta_u't}, & t\le0.
\end{cases}
\]
Finally set \[
\|v\|_b^{\mathrm{str}} := \|P(b)v\|_{s,b}^{\mathrm{str}} + \|Q(b)v\|_{u,b}^{\mathrm{str}}.
\]

The finiteness and continuity of these norms are proven analogously to the adapted norms of \cref{thm:ED-NUED-equiv}. In fact also the proof of strong ED of \( \Phi\) carries over using estimates of the form \[
    e^{- \overline{\delta}_s' \tau} \|\Phi(-\tau,b)x\| \leq D N(b) e^{- \left( \overline{\delta}_s' - \overline{\delta}_s - \varepsilon \right) \tau } \|x\|  
\] for \( x \in E^{s}_b\) and \( \tau \geq 0\), and the analog for the unstable direction. These bounds suffice to obtain analogous tail estimates to the ones proven in \cref{thm:ED-NUED-equiv}. 
\end{proof}

\begin{remark}\label{rem:strongLyap-known}
  We want to remark at this point that such a theorem as above is by no means surprising. In fact the argumentation done in~\cite{BarreiraValls2011SimpleProof} boils down to the same thought. Our extension here only really remarks that this formalism is applicable more broadly and is fundamentally changing the geometry of the space we are working on. This, again, is the strength of the Banach bundle formalism. It not only allows us to give a broader definition of non-autonomous dynamics, but also turns these known constructions into geometric transformations intrinsic to the class of Banach bundles.
\end{remark}

\section{Bundle Hartman--Grobman theorems}\label{sec:HGT}

There have been several Hartman--Grobman type theorems beyond the classical setting. For example Bates and Lu~\cite{BatesLu1994HartmanGrobman} obtain a result for autonomous evolution equations on a fixed Hilbert space equipped with an inertial manifold, Pötzsche and Russ~\cite{PotzscheRuss2016TopologicalDecoupling} work on two-parameter semigroups which are generated by families of unbounded operators on a fixed Banach space by utilizing the dichotomy spectrum explicitly and Coayla-Terán, Mohammed, and Ruffino~\cite{Coayla-TeranEtAl2007HartmanGrobman} are laying out a linearization result for random dynamical systems, compare~\cite{Arnold1998RandomDynamical}, which is deeply related to our work, as their \emph{random norms} provide utility analogously to our renorming procedure. Backes and Dragičević provide a similar result in the discrete time case~\cite{BackesDragicevic2021LinearizationInfinitedimensional}. 
Furthermore, Backes, Dragičević, and Palmer~\cite{BackesEtAl2021LinearizationHolder} have given Hölder linearization results of non-autonomous systems. 

In this section we will prove a Hartman--Grobman type theorem on general Banach bundles which allows for fiberwise-continuous, orbitwise measurable perturbations of the linear skew-product flow, to allow for the theorem to be applied to control theory in \cref{sec:App}, building substantially on the work of Barreira and Valls~\cite{BarreiraValls2011SimpleProof}.

For this it is necessary to broaden the class of continuous bounded bundle maps to a slightly more general class. 

\begin{definition}[Carath\'eodory enlargement]\label{def:BSpace}
\label{def:Caratheodory-enlargement}
Let \((\mathcal{E},\pi,B)\) be a Banach bundle over a base flow \(\theta_t b=b\cdot t\), and let \(\Phi\) be a linear skew-product flow on \(\mathcal{E}\).  We denote by \(\mathcal B\) the space of all maps \[
  \eta: \mathcal{E} \to \mathcal{E}
\]
such that:
\begin{enumerate}[label=\textup{(\roman*)}]
  \item \(\eta\) is fiber-preserving: \[
    \pi(\eta(v))=\pi(v),
  \] i.e.\ a \emph{bundle map};
  \item \(\eta\) vanishes on the zero section: \[
    \eta(0_b)=0_b \qquad (b\in B);
  \]
  \item \(\eta\) is bounded: \[
    \|\eta\|_{\mathcal B} := \sup_{v\in \mathcal{E}}\|\eta(v)\| <\infty;
  \]
  \item \(\eta\) is fiberwise continuous, i.e. \[
    \eta|_{E_b}:E_b\to E_b
  \]
  is continuous for every \(b\in B\);
  \item \(\eta\) is orbitwise strongly measurable, i.e.\ for every
  \(v\in E_b\), the map \[
    t\longmapsto \Phi_{-t}\eta(\Phi_t v)
  \]
  is strongly measurable as an \(E_b\)-valued map.
\end{enumerate}
\end{definition}

Similarly to the standard result that the space of bounded measurable maps is a Banach space when equipped with the sup-norm, we have the following.

\begin{lemma}[\(\mathcal{B}\) is Banach]
\((\mathcal{B},\|\cdot\|_{\mathcal{B}})\) is a Banach space.
\end{lemma}

\begin{proof}
It is immediate from the fiberwise linear structure that \(\mathcal B\) is a normed vector space.  Indeed, fiber-preserving maps may be added and scaled fiberwise; boundedness, vanishing on the zero section, fiberwise continuity, and orbitwise strong measurability are preserved under finite linear combinations.

Let \((\eta_n)\) be Cauchy in \(\mathcal B\).  For each \(v\in E_b\), the sequence \((\eta_n(v))\) lies in the single Banach space \(E_b\) and satisfies \[
  \|\eta_n(v)-\eta_m(v)\| \leq \|\eta_n-\eta_m\|_{\mathcal B}.
\]
Hence, \((\eta_n(v))\) converges in \(E_b\). Define \[
  \eta(v):=\lim_{n\to\infty}\eta_n(v).
\]
Then \(\eta\) is fiber-preserving and \(\eta(0_b)=0_b\).

We claim that \(\eta_n\to\eta\) in \(\mathcal B\)-norm.  Given \(\varepsilon>0\), choose \(N\) such that \[
  \|\eta_n-\eta_m\|_{\mathcal B}<\varepsilon \qquad(n,m\geq N).
\]
Fix \(n\geq N\).  Passing to the limit \(m\to\infty\) in each fiber gives \[
  \|\eta_n(v)-\eta(v)\|\leq \varepsilon \qquad(v\in \mathcal{E}).
\]
Taking the supremum over \(v\in \mathcal{E}\), we obtain \[
  \|\eta_n-\eta\|_{\mathcal B}\leq \varepsilon \qquad(n\geq N).
\]
Thus, \(\eta_n\to\eta\) uniformly in fiber norm.  In particular, \(\eta\) is bounded.

It remains to check that \(\eta\in\mathcal B\). Fiberwise continuity follows because, for each \(b\in B\), \[
  \eta_n|_{E_b}\to\eta|_{E_b}
\]
uniformly, and a uniform limit of continuous maps between Banach spaces is continuous.

Finally fix \(v\in E_b\).  For each \(n\), the map \[
  f_n(t):=\Phi_{-t}\eta_n(\Phi_t v)
\]
is strongly measurable as an \(E_b\)-valued map. Since \(\eta_n\to\eta\) pointwise and \(\Phi_{-t}:E_{b\cdot t}\to E_b\) is continuous linear for each fixed \(t\), \[
  f_n(t)\to f(t):=\Phi_{-t}\eta(\Phi_t v) \qquad\text{in }E_b.
\]
Thus, \(f\) is the pointwise limit of strongly measurable \(E_b\)-valued maps, hence is strongly measurable.  Therefore, \(\eta\) is orbitwise strongly measurable.

So \(\eta\in\mathcal B\), and \(\eta_n\to\eta\) in \(\|\cdot\|_{\mathcal B}\).  Hence, every Cauchy sequence in \(\mathcal B\) converges in \(\mathcal B\), proving completeness.
\end{proof}

\begin{lemma}\label{lem:C-closed}
    The space of bounded continuous bundle maps which vanish on the zero section \( \mathcal{C}\) is a closed subspace of \( \mathcal{B}\).
\end{lemma}
\begin{proof}
  This is a classic application of \cref{lem:conv-criterion}. Take \( \eta_N \in \mathcal{B}\) continuous and \( \eta_N \to \eta\) in \( \mathcal{B}\).
  Let \( v_k \to v_0\) in \( \mathcal{E}\) with basepoints \( b_k \to b_0\). Choose a local section \( \tau \) with \[
      \tau(b_0) = \eta(v_0)
  .\] Now it is clear, since \( \eta\) is a bundle map, that \[
      \pi(\eta(v_k)) = b_k \to b_0 = \pi(\eta(v_0))
  .\] Fix an index \( N \in \N\) and choose a local section \( \tau_N : W_N \to \mathcal{E}\) such that \[
      \tau_N(b_0)=\eta_N(v_0)
  \] and estimate
  \begin{align*}  
      &\|\eta(v_k) - \tau(\pi(\eta(v_k)))\| = \|\eta (v_k) - \tau(b_k)\| \\ 
      &\leq \|\eta(v_k) - \eta_N(v_k)\| + \|\eta_N(v_k) - \tau_N(b_k)\| + \| \tau_N(b_k) -\tau(b_k)\|
  \end{align*}   
  Since \( \eta_N\) is continuous, invoking \cref{lem:conv-criterion} shows that the second summand goes to \( 0\). Furthermore, we get \[
      \|\tau_N(b_k)- \tau(b_k)\| \to \|\tau_N(b_0)- \tau(b_0)\| = \|\eta_N(v_0)- \eta(v_0)\| \leq \|\eta_N -\eta\|_{\mathcal{B}}    
  ,\] which means that from the first and third summands we get \[ 
  \limsup_{k \to \infty } \|\eta(v_k) -\tau(b_k)\| \leq 2\|\eta_N -\eta\|_{\mathcal{B}}
  .\] Now by assumption \( \eta_N \to \eta\) in \( \mathcal{B}\) for \( N \to \infty \), meaning we get \[
      \|\eta(v_k) -\tau(b_k)\| \to 0 
  ,\] 
  which by \cref{lem:conv-criterion} gives that \( \eta(v_k) \to \eta(v_0)\) as desired.
\end{proof}

\begin{definition}[Carath\'eodory perturbations]\label{def:CarPert}
We call a family of fiber-preserving maps \( \mathcal{R}_b(t,\cdot ) : E_{b \cdot  t} \to E_{b \cdot t}\) a \textit{Carath\'eodory perturbation} if it satisfies the following conditions:
\begin{enumerate}[label=\textup{(\roman*)}]
  \item \emph{measurability:} for every curve \( t \mapsto y(t) \in E_{b \cdot t}\) whose pullback \( t \mapsto \Phi_{-t}y(t) \in E_b\) is strongly measurable, the pullback \( t \mapsto \Phi_{-t}\mathcal{R}_b(t,y(t))\) is strongly measurable as well;
  \item \emph{coherence:} \( \mathcal{R}\) respects the skew-product structure, i.e. \[ \mathcal{R}_{b \cdot s}(t,\cdot ) = \mathcal{R}_b(t+s,\cdot)\] for almost all \( t\) and all \( s\);
  \item \emph{zero section:} \( \mathcal{R}_b(t,0)=0\) for all \( t\);
  \item \emph{Lipschitz condition:} \( \mathcal{R}\) is uniformly Lipschitz, i.e. \[
      \|\mathcal{R}_b(t,v)-\mathcal{R}_b(t,w)\| \leq  \delta_0 \min \{1,\|v-w\| \}
  \] for almost all \( t\).
\end{enumerate}
\end{definition}

\begin{example}\label{ex:pointwise-source}
  In the case of non-autonomous dynamical systems these Carath\'eodory perturbations are usually of the form \( \mathcal{R}_{b}(t,\cdot ) := f|_{E_{\theta_t b}}\) for a continuous bundle map \( f : \mathcal{E} \to \mathcal{E}\). 
\end{example}

\begin{lemma}\label{lem:bdd-orbit-zero}
If \( \Phi\) admits an exponential dichotomy, \( v \in \mathcal{E}\) and  \[
  \sup_{t\in\R}\|\Phi_t v\|<\infty,
\]
then \(v=0\).
\end{lemma}

\begin{proof}
Set \(b:=\pi(v)\) and \[
  M:=\sup_{t\in\R}\|\Phi_t v\|<\infty.
\]
By invariance of \(Q\), for \(t\geq0\), \[
  Q(b)v=\Phi_{-t}Q(b\cdot t)\Phi_t v.
\]
Hence, the unstable ED estimate gives \[
  \|Q(b)v\|\le De^{-\delta t}\|\Phi_t v\|\le DMe^{-\delta t}\to0,
\]
thus \(Q(b)v=0\). Similarly, \[
  P(b)v=\Phi_tP(b\cdot(-t))\Phi_{-t}v
\]
and the stable ED estimate give \[
  \|P(b)v\|\le De^{-\delta t}\|\Phi_{-t}v\|\le DMe^{-\delta t}\to0.
\]
Thus, \(P(b)v=0\), and therefore \[
  v=P(b)v+Q(b)v=0.
\]
\end{proof}

\begin{lemma}\label{lem:KomGrowth}
  Assuming \( \Phi\) is a linear skew-product flow, we have for every compact subset \( K \subset \R \times B\) \[
      \sup_{(t,c) \in K}\|\Phi(t,c)\| < \infty  
  .\] In particular for every \( b \in B\) and every \( T > 0\) we have\[
      M_T(b):= \sup \{\|\Phi(t,\theta_sb)\| \ |\ |t|\leq T, |s| \leq T   \} < \infty 
  .\]  
\end{lemma}
\begin{proof}
Assume on the contrary that there exist \( (t_n,b_n) \subset K\) and \( v_n \in E_{b_n}\) such that \( \|v_n\| = 1 \) and \[
    \|\Phi(t_n,b_n)v_n\| \to \infty 
.\] If necessary passing to a subsequence, assume \[
    (t_n,b_n) \to (t,b)
\] and set \[
    u_n := \frac{v_n}{\|\Phi(t_n,b_n)v_n\|}
.\] Then it's clear that \( \|u_n\| \to 0 \) and \( \pi(u_n)=b_n \to b\). By the convergence condition for Banach bundles we have that \[
    u_n \to 0_{b}
.\] Thus, by continuity of \( \Phi\), we must have \[
    \Phi_{t_n}u_n \to \Phi_t 0_b = 0_{\theta_t b} 
,\] but of course by assumption \[
    \|\Phi_{t_n}u_n\|  = \frac{\|\Phi(t_n,b_n)v_n\| }{\|\Phi(t_n,b_n)v_n\| } = 1 
,\] which is a contradiction.      
\end{proof}
Next we will provide an argument using \cref{lem:KomGrowth} and a Picard--Lindelöf type approach showing the uniqueness of cocycles which satisfy a variation-of-constants formula and its continuity implications.

\begin{proposition}
\label{prop:cutoff-flow}
Let \(\mathcal{R}\) be a Carath\'eodory perturbation with constant \(\delta_0\). Then there is exactly one map
\(\Psi:\R \times \mathcal{E} \to \mathcal{E}\) with \(\pi(\Psi_t(v))=\pi(v)\cdot t\) and
\begin{equation}
\label{eq:voc}
  \Psi_t(v)=\Phi_tv+\int_0^t\Phi_{t-s}\,\mathcal{R}_b\bigl(s,\Psi_s(v)\bigr)\,ds
  \qquad(t\in\R,\ v\in E_b),
\end{equation}
which, assuming it is continuous, is a skew-product flow over \(\theta\). Moreover, for every \(v\in E_b\), the pulled-back orbit \(t\mapsto\Phi_{-t}\Psi_t(v)\) belongs to \(C([-T,T],E_b)\) for every \(T>0\). Also, for every \(b\in B\), \(T>0\), and \(|t|\le T\), the map \[
  \Psi_t|_{E_b}:E_b\to E_{b\cdot t}
\]
is continuous.
\end{proposition}

\begin{proof}
Fix \(v\in E_b\) and set \(y(t):=\Phi_{-t}\Psi_t(v)\in E_b\) for a map
\(\Psi_t(v)\in E_{b\cdot t}\). Applying \(\Phi_{-t}\) shows \eqref{eq:voc} is equivalent to
\begin{equation}
\label{eq:pullback}
  y(t)=v+\int_0^tG(s,y(s))\,ds
\end{equation}
where \( G(s,w):=\Phi_{-s}\,\mathcal{R}_b(s,\Phi_sw)\in E_b\), since \( \Psi_s(v) = \Phi_s y(s)\). 

Fix \(T>0\), \(y\in C([-T,T],E_b)\). By \cref{def:CarPert}, \(s\mapsto G(s,y(s))\) is strongly measurable, and by \cref{lem:KomGrowth} we obtain the uniform estimates \begin{align*}
  \|G(s,x)\| &\leq M_T(b)\delta_0, \\ 
  \|G(s,x)-G(s,y)\|&\leq M_T(b)^{2}\delta_0\|x-y\|
\end{align*}
for almost every \( s \in  [-T,T]\). Thus, \(G(\cdot,y(\cdot))\) is essentially bounded, so \((\mathcal Ty)(t):=v+\int_0^tG(s,y(s))\,ds\) maps \(C([-T,T],E_b)\) into itself, because indefinite Bochner integrals of bounded maps are automatically Lipschitz, hence continuous. 

To show that \( \mathcal{T}\) admits a unique fixed point, choose \(L:=M_T(b)^{2}\delta_0\) and renorm the space \( C([-T,T],E_b)\) equivalently with \[
    \|y\|_\ast:=\sup_{|t|\le T}e^{-2L|t|}\|y(t)\|
.\] Using this norm it is easy to calculate 
\begin{align*}
  \|(\mathcal Ty)(t)-(\mathcal T\tilde y)(t)\| &\leq\Bigl|\int_0^tL\|y(s)-\tilde y(s)\|\,ds\Bigr| \\
  &\leq L\|y-\tilde y\|_\ast\Bigl|\int_0^te^{2L|s|}ds\Bigr| \\
  &\leq\tfrac12e^{2L|t|}\|y-\tilde y\|_\ast,
\end{align*}
so \(\mathcal T\) is a contraction and \eqref{eq:pullback} has a unique solution, say \(y_T\), on \([-T,T]\). Solutions for different \(T\) naturally agree on overlaps by uniqueness. Hence, glueing solutions \(y\) on all of \(\R\) together allows us to define \(\Psi_t(v):=\Phi_ty(t)\). The same estimate gives continuous dependence on the initial value. Indeed, if \(y_v,y_w\) solve \eqref{eq:pullback} with initial data \(v,w\in E_b\), then, with \(L:=M_T(b)^2\delta_0\), Gronwall's lemma gives \[
  \|y_v(t)-y_w(t)\|\leq e^{L|t|}\|v-w\|,\qquad |t|\leq T
.\]
Since \(\Psi_t(v)=\Phi_t y_v(t)\), it follows that \(\Psi_t|_{E_b}:E_b\to E_{b\cdot t}\) is continuous for each fixed \(t\).

For fixed \(s\), both \(t\mapsto\Psi_{s+t}(v)\) and \(t\mapsto\Psi_t(\Psi_s(v))\) solve \eqref{eq:voc} with the same initial value over the basepoint \(b\cdot s\) and \(\Psi_0=\id\) is immediate from \eqref{eq:voc}. This finishes the proof.
\end{proof}

\begin{remark}[Standing assumptions for the rest of the section]\label{rem:StanAss}
  Throughout the rest of the section we fix the standing assumptions: 
  \begin{enumerate}
    \item The linear SPF \( \Phi\) admits an exponential dichotomy with projections \( P\) and \( Q\); in \cref{subsec:Holder-HG} this dichotomy is strong. 
    \item We are assuming a Carath\'eodory perturbation \( \mathcal{R}\) as in \cref{def:CarPert} and a non-linear cocycle \( \Psi\) over \( \theta\) that satisfies the \textit{Variation of Constants}-formula \[
       \Psi_t(x)=\Phi_t x+\int_0^t\Phi_{t-s}\mathcal{R}_b(s,\Psi_s(x))\,ds,\qquad x\in E_b 
    .\] In \cref{subsec:local-HG}, this is weakened to a local condition. If this cocycle is continuous, it is a proper skew-product flow and in particular this will always be the case when the continuity conditions \[
        F(\mathcal{C}) \subset \mathcal{C}, \qquad \zeta \in \mathcal{C}
    \] are satisfied, which they will be in all explored examples.
  \end{enumerate}
  
\end{remark}

\subsection{Global Hartman--Grobman theorem}

\begin{definition}[Fixed-point operator]\label{def:fixed-point-operator}
Given an \(\eta\in\mathcal B\), define
\begin{align*}
  (F\eta)(v):= &\int_0^\infty \Phi_\sigma P\,\mathcal{R}_b(-\sigma,(\id+\eta)(\Phi_{-\sigma}v))\,d\sigma \\
  -&\int_0^\infty \Phi_{-\sigma}Q\,\mathcal{R}_b(\sigma,(\id+\eta)(\Phi_\sigma v))\,d\sigma,
  \qquad v\in E_b.
\end{align*}
\end{definition}

\begin{lemma}[Well-definedness of \(F\)]\label{lem:F-welldef}
For every \(\eta\in\mathcal B\), one has \(F\eta\in\mathcal B\), and
\[
  \|F\eta\|_{\mathcal B}\le \frac{2D\delta_0}{\delta}.
\]
\end{lemma}

\begin{proof}
Fix \(\eta\in\mathcal B\), write \[
  H_\eta:=\id+\eta,
\]
and fix \(v\in E_b\). Put \[
  r_{\eta,v}(t):=\mathcal{R}_b(t,H_\eta(\Phi_t v)).
\]
We first check measurability of this orbit source. Since \[
  \Phi_{-t}H_\eta(\Phi_t v)=v+\Phi_{-t}\eta(\Phi_t v),
\]
and since \(\eta\in\mathcal B\), the pullback \(t\mapsto\Phi_{-t}H_\eta(\Phi_t v)\) is strongly measurable, and so is, by the Carath\'eodory measurability assumption, \[
  t\mapsto\Phi_{-t}r_{\eta,v}(t).
\]
Hence, the integrals of \( F\) make sense, and we can estimate their norm by \[
    \|F\eta(v)\| \leq \frac{2D}{\delta} \|r_{\eta,v}\|  
\] by using the norm estimate for Bochner integrals and applying \cref{eq:ED-stable,eq:ED-unstable}. Moreover, since \(\mathcal{R}\) is bounded by \( \delta_0\), so is \( r_{\eta,v}(t)\) for a.e.\ \(t\). Hence, both integrals defining \((F\eta)(v)\) exist in \(E_b\), and \[
  \|(F\eta)(v)\| \leq \frac{2D\delta_0}{\delta}
,\] which proves the boundedness. 

We next prove fiberwise continuity. Fix \(b\in B\) and let \(v_n\to v\) in \(E_b\). For every fixed \(\sigma\geq0\), \[
  H_\eta(\Phi_{-\sigma}v_n)\to H_\eta(\Phi_{-\sigma}v),\qquad
  H_\eta(\Phi_\sigma v_n)\to H_\eta(\Phi_\sigma v),
\]
because \(\eta\) is fiberwise continuous. Since \(\mathcal{R}_b(t,\cdot)\) is Lipschitz for a.e.\ \(t\), the corresponding source terms converge for a.e.\ \(\sigma\). The stable and unstable differences are each dominated by \[
  D\delta_0 e^{-\delta\sigma}
,\]
hence dominated convergence gives \[
  (F\eta)(v_n)\to(F\eta)(v)
\]
in \(E_b\). Thus, \(F\eta\) is fiberwise continuous.

It remains to verify orbitwise strong measurability. By a linear change of variables and the coherence condition in \cref{def:CarPert}, we get \[
  \Phi_{-t}(F\eta)(\Phi_t v)=\int_{-\infty}^t \Phi_{-u}P\,r_{\eta,v}(u)\,du-\int_t^\infty \Phi_{-u}Q\,r_{\eta,v}(u)\,du.
\]
The two integrands are strongly measurable, locally integrable \(E_b\)-valued maps with exponentially decaying tails by \cref{eq:ED-stable,eq:ED-unstable} and the boundedness of \( r\). Thus, the right-hand side is strongly measurable as well, so \(F\eta\) is orbitwise strongly measurable, and so \[
  F\eta\in\mathcal B.
\]
\end{proof}

\begin{lemma}[Fixed point and conjugacy]\label{lem:fixedpoint-conjugacy}
Let \(\eta\in\mathcal B\) and set \(H:=\id+\eta\). Then \(\eta\) is a fixed point of \(F\) if and only if \[
  H(\Phi_t v)=\Psi_t(H(v))\qquad(t\in\R,\ v\in \mathcal{E}).
\]
\end{lemma}

\begin{proof}
Write \(F\eta=(F\eta)^s+(F\eta)^u\), where \[
  (F\eta)^s(v):=\int_0^\infty\Phi_\sigma P\,\mathcal{R}_b(-\sigma,H(\Phi_{-\sigma}v))\,d\sigma,
\]
and \[
  (F\eta)^u(v):=-\int_0^\infty\Phi_{-\sigma}Q\,\mathcal{R}_b(\sigma,H(\Phi_\sigma v))\,d\sigma.
\]
Using the cocycle identity and coherence of \(\mathcal{R}\), the changes of variables \(u=t-\sigma\), \(u=-\sigma\), \(u=t+\sigma\), and \(u=\sigma\) yield
\begin{align*}
  (F\eta)^s(\Phi_t v)&=\int_{-\infty}^t\Phi_{t-u}P\,\mathcal{R}_b(u,H(\Phi_{u}v))\, du, \\
  \Phi_t(F\eta)^s(v)&=\int_{-\infty}^0\Phi_{t-u}P\,\mathcal{R}_b(u,H(\Phi_{u}v))\,du,
\end{align*}
and
\begin{align*}
  (F\eta)^u(\Phi_t v)&=-\int_t^\infty\Phi_{t-u}Q\,\mathcal{R}_b(u,H(\Phi_{u}v))\,du, \\ 
  \Phi_t(F\eta)^u(v)&=-\int_0^\infty\Phi_{t-u}Q\,\mathcal{R}_b(u,H(\Phi_{u}v))\,du.
\end{align*}
Subtracting yields, with oriented integrals if \(t<0\),
\begin{equation}\label{eq:Fix1}
  (F\eta)(\Phi_t v)-\Phi_t(F\eta)(v)=\int_0^t\Phi_{t-u}\mathcal{R}_b(u,H(\Phi_u v))\,du.
\end{equation}

If \(\eta=F\eta\), then \eqref{eq:Fix1} gives \[
  H(\Phi_t v)=\Phi_tH(v)+\int_0^t\Phi_{t-u}\mathcal{R}_b(u,H(\Phi_u v))\,du.
\]
Thus, \(t\mapsto H(\Phi_t v)\) and \(t\mapsto\Psi_t(H(v))\) solve the same variation-of-constants equation with the same initial value. Now let \[ 
Z(t):= \Phi_{-t}(H(\Phi_t v)- \Psi_t(H(v))). 
\]
Fix \( T>0\). Using \cref{lem:KomGrowth} we find a constant \( C=C(T)\) such that, for \( 0\leq t\leq T\), \[
    \|Z(t)\| \leq C \delta_0 \int_{0}^{t} \|Z(u)\| \, du  
\] by subtracting the variation of constants formulas and the Lipschitz bound \( \delta_0\) of \( \mathcal{R}\). Gronwall now yields \( Z = 0 \) on \( [0,T]\) and, with analogous argumentation, on \( [-T,0]\). Since \( T\) is arbitrary we have \( H(\Phi_tv) = \Psi_t(H(v))\) for all \( t \in \R\).

Conversely, assume the conjugacy identity. Applying the variation-of-constants formula to \(\Psi_t(H(v))\), and using \(H(\Phi_u v)=\Psi_u(H(v))\), gives \[
  \eta(\Phi_t v)-\Phi_t\eta(v)=\int_0^t\Phi_{t-u}\mathcal{R}_b(u,H(\Phi_u v))\,du.
\]
Comparing with \eqref{eq:Fix1}, and setting \[
  \xi:=\eta-F\eta,
\]
we obtain \[
  \xi(\Phi_t v)=\Phi_t\xi(v)\qquad(t\in\R).
\]
Since \(\eta\in\mathcal B\) and \(F\eta\in\mathcal B\) by \cref{lem:F-welldef}, the map \(\xi\) is bounded. Hence, \[
  \sup_{t\in\R}\|\Phi_t\xi(v)\|=\sup_{t\in\R}\|\xi(\Phi_t v)\|\le\|\xi\|_{\mathcal B}<\infty.
\]
By \cref{lem:bdd-orbit-zero}, \(\xi(v)=0\). Since \(v\) was arbitrary, \(\xi=0\), and therefore \[
  \eta=F\eta.
\]
\end{proof}

\begin{lemma}[Contraction]\label{lem:F-contraction}
Assume
\[
  \delta_0<\frac{\delta}{2D}.
\]
Then \(F:\mathcal B\to\mathcal B\) is a contraction. More precisely, for all \(\eta,\xi\in\mathcal B\),
\[
  \|F\eta-F\xi\|_{\mathcal B}\le \frac{2D\delta_0}{\delta}\|\eta-\xi\|_{\mathcal B}.
\]
\end{lemma}

\begin{proof}
By \cref{lem:F-welldef}, \(F\eta,F\xi\in\mathcal B\). Fix \(v\in E_b\) and write
\[
  H_\eta:=\id+\eta,\qquad H_\xi:=\id+\xi.
\]
For the stable branch, the Lipschitz condition (\cref{def:CarPert}) gives, for a.e.\ \(\sigma\geq 0\),
\[
  \|\mathcal{R}_b(-\sigma,H_\eta(\Phi_{-\sigma}v))-\mathcal{R}_b(-\sigma,H_\xi(\Phi_{-\sigma}v))\|\le \delta_0\|\eta-\xi\|_{\mathcal B}.
\]
Using \eqref{eq:ED-stable}, \[
  \left\|\Phi_\sigma P\big[\mathcal{R}_b(-\sigma,H_\eta(\Phi_{-\sigma}v))-\mathcal{R}_b(-\sigma,H_\xi(\Phi_{-\sigma}v))\big]\right\| \leq D\delta_0e^{-\delta\sigma}\|\eta-\xi\|_{\mathcal B}
.\]
Similarly, using \eqref{eq:ED-unstable}, \[
  \left\|\Phi_{-\sigma}Q\big[\mathcal{R}_b(\sigma,H_\eta(\Phi_\sigma v))-\mathcal{R}_b(\sigma,H_\xi(\Phi_\sigma v))\big]\right\| \leq D\delta_0e^{-\delta\sigma}\|\eta-\xi\|_{\mathcal B}
.\]
Therefore, \[
  \|(F\eta)(v)-(F\xi)(v)\|\le 2D\delta_0\|\eta-\xi\|_{\mathcal B}\int_0^\infty e^{-\delta\sigma}\,d\sigma =\frac{2D\delta_0}{\delta}\|\eta-\xi\|_{\mathcal B}.
\]
Taking the supremum over \(v\in \mathcal{E}\) gives
\[
  \|F\eta-F\xi\|_{\mathcal B}\le \frac{2D\delta_0}{\delta}\|\eta-\xi\|_{\mathcal B}.
\]
Since \(2D\delta_0/\delta<1\), \(F\) is a contraction.
\end{proof}

\begin{theorem}[Hartman--Grobman theorem]\label{thm:HG}
Given our standing assumptions \cref{rem:StanAss} and additionally that \[
  \delta_0<\frac{\delta}{2D},
\]
there exists a unique \(\eta^*\in\mathcal B\) such that \(H:=\id+\eta^*\) satisfies \[
  H(\Phi_t v)=\Psi_t(H(v))\qquad(t\in\R,\ v\in \mathcal{E}).
\]
Moreover, \(H\) restricts to a homeomorphism on every fiber. Its inverse is of the form \[
  H^{-1}=\id+\zeta
,\]
where\[
  \zeta(v):=-\int_0^\infty\Phi_\sigma P\,\mathcal{R}_b(-\sigma,\Psi_{-\sigma}(v))\,d\sigma+\int_0^\infty\Phi_{-\sigma}Q\,\mathcal{R}_b(\sigma,\Psi_\sigma(v))\,d\sigma
\]
for \( v \in \mathcal{E}\).
The map \(\zeta\) is bounded, fiber-preserving, vanishes on the zero section, and is fiberwise continuous.

If, in addition, \[
  F(\mathcal C)\subseteq\mathcal C,\qquad \zeta\in\mathcal C,
\]
then \(\eta^*,\zeta\in\mathcal C\). Consequently, \(H\) and \(H^{-1}\) are continuous bundle maps and \( \Psi\) is a continuous SPF.
\end{theorem}

\begin{proof}
Set \[
  q:=\frac{2D\delta_0}{\delta}.
\]
By assumption \(q<1\). By \cref{lem:F-contraction}, \(F:\mathcal B\to\mathcal B\) is a contraction with contraction constant \(q\). Since \(\mathcal B\) is complete, Banach's fixed-point theorem gives a unique \(\eta^*\in\mathcal B\) such that \[
  F\eta^*=\eta^*.
\]
By \cref{lem:fixedpoint-conjugacy}, \(H:=\id+\eta^*\) satisfies \[
  H(\Phi_t v)=\Psi_t(H(v))\qquad(t\in\R,\ v\in \mathcal{E}).
\]
We first prove fiberwise injectivity. Suppose \(v,w\in E_b\) and \[
  H(v)=H(w).
\]
Using the conjugacy identity gives \[
  H(\Phi_t v)=\Psi_t(H(v))=\Psi_t(H(w))=H(\Phi_t w)\qquad(t\in\R).
\]
Hence, \[
  \Phi_t(v-w)=\eta^*(\Phi_t w)-\eta^*(\Phi_t v).
\]
The right-hand side is uniformly bounded in \(t\), since \(\eta^*\in\mathcal B\). Thus, \[
  \sup_{t\in\R}\|\Phi_t(v-w)\|<\infty.
\]
By \cref{lem:bdd-orbit-zero}, \(v-w=0\). Hence, \(H|_{E_b}\) is injective.

Next, we construct the inverse correction. For \(v\in E_b\), define \[
  \zeta(v):=-\int_0^\infty\Phi_\sigma P\,\mathcal{R}_b(-\sigma,\Psi_{-\sigma}(v))\,d\sigma+\int_0^\infty\Phi_{-\sigma}Q\,\mathcal{R}_b(\sigma,\Psi_\sigma(v))\,d\sigma.
\]
By Proposition~\ref{prop:cutoff-flow}, the pulled-back orbit
\(t\mapsto\Phi_{-t}\Psi_t(v)\) is continuous as an \(E_b\)-valued curve.
Hence, Definition~\ref{def:CarPert} applies to \(t\mapsto\Psi_t(v)\), and
the pulled-back source terms are strongly measurable. Thus, the Green
integrands defining \(\zeta(v)\) are strongly measurable. Now \cref{eq:ED-stable,eq:ED-unstable}, and \[
  \|\mathcal{R}_b(t,y)\|\leq\delta_0
\]
give \[
  \|\zeta(v)\|\leq 2D\delta_0\int_0^\infty e^{-\delta\sigma}\,d\sigma=\frac{2D\delta_0}{\delta}.
\]
Thus, \(\zeta\) is bounded and fiber-preserving. Also, \(\zeta(0_b)=0_b\), since \(\mathcal{R}_b(t,0)=0\) and the zero section is a solution of the variation-of-constants equation.

We now prove that \(\zeta\) is fiberwise continuous. Fix \(b\in B\) and let \(v_n\to v\) in \(E_b\). For every fixed \(\sigma\geq0\), the fiber maps \(\Psi_{\pm\sigma}|_{E_b}\) are continuous by \cref{prop:cutoff-flow}. Hence, \[
  \Psi_{-\sigma}(v_n)\to\Psi_{-\sigma}(v),\qquad \Psi_\sigma(v_n)\to\Psi_\sigma(v)
\]
in the corresponding fibers. Since \(\mathcal{R}_b(t,\cdot)\) is Lipschitz for almost every \(t\), we have, for almost every \(\sigma\geq0\), \[
  \mathcal{R}_b(-\sigma,\Psi_{-\sigma}(v_n))\to\mathcal{R}_b(-\sigma,\Psi_{-\sigma}(v)),
\]
and similarly \[
  \mathcal{R}_b(\sigma,\Psi_\sigma(v_n))\to\mathcal{R}_b(\sigma,\Psi_\sigma(v)).
\]
For the stable branch, the difference of the integrands is bounded by \[
  De^{-\delta\sigma}\delta_0\min\{1,\|\Psi_{-\sigma}(v_n)-\Psi_{-\sigma}(v)\|\}\leq D\delta_0e^{-\delta\sigma},
\]
and the right-hand side is integrable on \([0,\infty)\). Dominated convergence gives convergence of the stable integral. The unstable branch is analogous. Thus, \[
  \zeta(v_n)\to\zeta(v)
\]
in \(E_b\). Therefore, \(\zeta|_{E_b}\) is continuous for every \(b\in B\).

Put \[
  k:=\id+\zeta.
\]
We claim that \(k\) conjugates \(\Psi\) back to \(\Phi\). Repeating the same change-of-variables calculation as in the proof of \cref{lem:fixedpoint-conjugacy}, now along the nonlinear orbit, gives \[
  \zeta(\Psi_t v)-\Phi_t\zeta(v)=-\int_0^t\Phi_{t-s}\mathcal{R}_b(s,\Psi_s(v))\,ds.
\]
Adding the variation-of-constants formula yields \[
  k(\Psi_t v)=\Phi_t k(v)\qquad(t\in\R,\ v\in \mathcal{E})
.\]
We prove first that \(k\circ H=\id\). Define \[
  \beta:=k\circ H-\id=\eta^*+\zeta\circ H.
\]
Using \(H\circ\Phi_t=\Psi_t\circ H\) and \(k\circ\Psi_t=\Phi_t\circ k\), we get \[
  k\circ H\circ\Phi_t=k\circ\Psi_t\circ H=\Phi_t\circ k\circ H.
\]
Therefore, \[
  \beta(\Phi_t v)=\Phi_t\beta(v)\qquad(t\in\R,\ v\in \mathcal{E}).
\]
Since \(\eta^*\) and \(\zeta\) are bounded, \(\beta\) is bounded. Hence, \[
  \sup_{t\in\R}\|\Phi_t\beta(v)\|=\sup_{t\in\R}\|\beta(\Phi_t v)\|<\infty.
\]
By \cref{lem:bdd-orbit-zero}, \(\beta(v)=0\). Since \(v\) was arbitrary, \[
  k\circ H=\id
.\]
It remains to prove \(H\circ k=\id\). Define \[
  \alpha:=H\circ k-\id=\zeta+\eta^*\circ k,\qquad A:=\sup_{v\in \mathcal{E}}\|\alpha(v)\|
.\]
The number \(A\) is finite because \(\eta^*\) and \(\zeta\) are bounded. Composing the two conjugacy identities gives \[
  H\circ k\circ\Psi_t=H\circ\Phi_t\circ k=\Psi_t\circ H\circ k
,\]
which means that \[
  (\id+\alpha)\circ\Psi_t=\Psi_t\circ(\id+\alpha)
.\]
Subtracting the variation-of-constants formula for \(\Psi_t(v)\) from the one for \(\Psi_t(v+\alpha(v))\), and applying the same estimates as above, gives \[
  \|\alpha(v)\|\leq 2D\delta_0\int_0^\infty e^{-\delta\sigma}\min\{1,A\}\,d\sigma=\frac{2D\delta_0}{\delta}\min\{1,A\}
.\]
Taking the supremum over \(v\in \mathcal{E}\), we obtain \[
  A\leq q\min\{1,A\}.
\]
Since \(q<1\), this forces \(A=0\). Hence, \[
  H\circ k=\id.
\]
Together with \(k\circ H=\id\), this proves that \(k=H^{-1}\).

Since \(\eta^*\in\mathcal B\), the map \(H|_{E_b}\) is continuous on every fiber. Since \(\zeta\) is fiberwise continuous, \(k|_{E_b}\) is continuous on every fiber. Therefore, \(H|_{E_b}\) is a homeomorphism with inverse \(k|_{E_b}\) for every \(b\in B\).
Finally assume \[
  F(\mathcal C)\subseteq\mathcal C,\qquad \zeta\in\mathcal C.
\]
Since \(0\in\mathcal C\), all Banach iterates \[
  \eta_{n+1}:=F\eta_n,\qquad \eta_0:=0,
\]
belong to \(\mathcal C\). They converge in \(\mathcal B\) to the unique fixed point \(\eta^*\). Since \(\mathcal C\) is closed in \(\mathcal B\), we get \[
  \eta^*\in\mathcal C.
\]
By the additional assumption \(\zeta\in\mathcal C\), both \[
  H=\id+\eta^*,\qquad H^{-1}=\id+\zeta
\]
are continuous bundle maps.
\end{proof}

\begin{remark}[On the additional continuity assumption]\label{rem:continuity-layer}
The core Carath\'eodory hypotheses imply that the inverse correction \(\zeta\) is fiberwise continuous, but they do not imply that \(\zeta\) is jointly continuous as a bundle map. Thus, the unconditional conclusion of \cref{thm:HG} is fiberwise homeomorphy, not joint continuity of \(H^{-1}\).
The additional continuity condition \[
  F(\mathcal C)\subseteq\mathcal C,\qquad \zeta\in\mathcal C
\]
is therefore a genuine assumption.

\end{remark}
The next proposition spells out the immediacy of the condition in the case where we are perturbing with a continuous bundle map.

\begin{proposition}\label{prop:ContForContBunMap}
  In the case where \( \mathcal{R}_b(t,\cdot ) = f|_{E_{b\cdot t}}\) for a jointly continuous bundle map \( f\) the conditions \[
      F(\mathcal{C}) \subset \mathcal{C}, \quad \zeta \in \mathcal{C}
  \] are immediate.
\end{proposition}

\begin{proof}
Let \(\eta\in\mathcal C\). Put \[
  H_\eta:=\id+\eta .
\]
Since \(\eta\in\mathcal C\), the map \(H_\eta\) is a jointly continuous bundle map over \(\id_B\). Hence, \[
  f\circ H_\eta:\mathcal{E} \to \mathcal{E}
\]
is again jointly continuous.

Fix \(v_0 \in E_{b_0}\). Choose, by \cref{prop:local-sections}, a continuous local section \[
  \tau:W\to \mathcal{E}, \qquad \tau(b_0)=v_0 .
\]
We first show that \[
  b\longmapsto (F\eta)(\tau(b))
\]
is a continuous local section on \(W\). We discuss the stable part; the unstable part is analogous. For \(\sigma\geq0\), set \[
  I^s(b,\sigma) := \Phi_\sigma P\, f\bigl(H_\eta(\Phi_{-\sigma}\tau(b))\bigr) \in E_b .
\]
For each finite \(T>0\), the map \[
  (b,\sigma)\longmapsto I^s(b,\sigma)
\]
is, by the continuity of \(\tau\), \(\Phi\), \(H_\eta\), and \(f\), together with continuity of the projection family \(P\) along sections, continuous on \(W\times[0,T]\). 

Define the truncated stable section \[
  \sigma_T^s(b):=\int_0^T I^s(b,\sigma)\,d\sigma .
\]
Fix \(b_1\in W\), and choose a compact neighborhood \(K\subset W\) of \(b_1\). On \(K\times[0,T]\), the integrand is, by compactness, uniformly continuous. Hence, Riemann sums for the above integral converge uniformly over \( K\). Each Riemann sum is a finite linear combination of continuous local sections, and is therefore continuous. By \cref{lem:uniform-limit} \(\sigma_T^s\) is continuous near \(b_1\). Since \(b_1\) was arbitrary, \(\sigma_T^s\) is continuous on \(W\).

The infinite tail is uniformly controlled by the exponential dichotomy: \[
  \sup_{b\in W} \left\| \int_T^\infty I^s(b,\sigma)\,d\sigma \right\| \leq D\delta_0\int_T^\infty e^{-\delta\sigma}\,d\sigma \to 0
\] as \( T \to \infty \).
Therefore, the full stable section \[
  b \mapsto \int_0^\infty \Phi_\sigma P\,f\bigl(H_\eta(\Phi_{-\sigma}\tau(b))\bigr)\,d\sigma
\]
is continuous. The unstable branch is treated in the same way. Hence, \[
  b\longmapsto (F\eta)(\tau(b))
\]
is a continuous local section through \((F\eta)(v_0)\).

We now prove continuity of \(F\eta\) at \(v_0\). Let \( v_n \to v_0\) in \( \mathcal{E}\) and set \(b_n:=\pi(v_n)\). Then \(b_n\to b_0\), and for all large \(n\), \(b_n\in W\). By \cref{lem:conv-criterion}, it is enough to show \[
  \|(F\eta)(v_n)-(F\eta)(\tau(b_n))\|\to 0.
\]
Again consider the stable branch. We have
\begin{align*}
  &\left\| \int_0^\infty \Phi_\sigma P \left[ f\bigl(H_\eta(\Phi_{-\sigma}v_n)\bigr) - f\bigl(H_\eta(\Phi_{-\sigma}\tau(b_n))\bigr) \right] \, d\sigma \right\| \\ 
  &\leq \int_0^\infty D e^{-\delta\sigma} \left\| f\bigl(H_\eta(\Phi_{-\sigma}v_n)\bigr) - f\bigl(H_\eta(\Phi_{-\sigma}\tau(b_n))\bigr) \right\|\,d\sigma .
\end{align*}

For each fixed \(\sigma\geq0\), continuity of \(\Phi\), \(\tau\), \(H_\eta\), and \(f\) gives \[
  f\bigl(H_\eta(\Phi_{-\sigma}v_n)\bigr) - f\bigl(H_\eta(\Phi_{-\sigma}\tau(b_n))\bigr) \longrightarrow 0
\]
in the corresponding fiber. Moreover, the integrand is bounded by \[
  D\delta_0 e^{-\delta\sigma},
\]
which is integrable on \([0,\infty)\). Dominated convergence gives convergence of the stable branch to zero. The unstable branch is the same, with \(\Phi_\sigma\) replaced by \(\Phi_{-\sigma}\) and vice versa. Thus, \[
  \|(F\eta)(v_n)-(F\eta)(\tau(b_n))\|\to0.
\]
Since \(b\mapsto(F\eta)(\tau(b))\) is a continuous local section through \((F\eta)(v_0)\), the convergence criterion yields \[
  (F\eta)(v_n)\to(F\eta)(v_0).
\]
Hence, \(F\eta\) is continuous at \(v_0\). Since \(v_0\) was arbitrary, \(F\eta\in\mathcal C\). Therefore, \[
  F(\mathcal C)\subseteq\mathcal C .
\]

Proving \( \zeta \in \mathcal{C}\) follows essentially the same argumentation using \cref{lem:conv-criterion}, the existence of local sections and the construction of truncated sections.
\end{proof}

\subsection{H\"older Hartman--Grobman theorem}\label{subsec:Holder-HG}

For \(\alpha\in(0,1]\) and \(\xi\in\mathcal B\) we write
\[
  [\xi]_\alpha:=\sup_b\sup_{\substack{v,w\in E_b\\v\neq w}}\frac{\|\xi(v)-\xi(w)\|}{\|v-w\|^\alpha}.
\]
Thus, \(\xi\) is fiberwise \(\alpha\)-H\"older if
\[
  [\xi]_\alpha<\infty .
\]

\begin{theorem}[H\"older Hartman--Grobman theorem]\label{thm:Holder-HG}
Assume the hypotheses of \cref{thm:HG}. Assume moreover that \(\Phi\) has a strong exponential dichotomy with data
\[
  (P,D,\delta_s,\bar\delta_s,\delta_u,\bar\delta_u).
\]
Set
\[
  \alpha_0:=\min\left\{\frac{\delta_s}{\bar\delta_s},\frac{\delta_u}{\bar\delta_u}\right\}.
\]
Let \(\alpha\in(0,\alpha_0)\), and assume 
\begin{equation}\label{eq:HoeldHGThm1}
  \delta_0<\frac{\min\{\delta_s-\alpha\bar\delta_s,\delta_u-\alpha\bar\delta_u\}}{4D^{1+3\alpha}}.
\end{equation}
Then the fixed-point correction \(\eta^*\) is fiberwise \(\alpha\)-H\"older. More precisely,
\[
  [\eta^*]_\alpha\le \frac{K_\alpha}{1-K_\alpha},
\]
where
\[
  K_\alpha:=\delta_0D^{1+2\alpha}\left[\frac{D^\alpha}{\delta_s-\alpha\bar\delta_s}+\frac{D^\alpha}{\delta_u-\alpha\bar\delta_u}+\frac{1}{\delta_s+\alpha\delta_u}+\frac{1}{\delta_u+\alpha\delta_s}\right].
\]
In particular,
\[
  \|\eta^*(v)-\eta^*(w)\|\le \frac{K_\alpha}{1-K_\alpha}\|v-w\|^\alpha
\]
for all \(v,w\) in the same fiber.

Moreover, if \(H^{-1}=\id+\zeta\), then \(\zeta\) is fiberwise \(\beta\)-H\"older for every
\[
  \beta<\alpha_0^-(\delta_0):=\min\left\{\frac{\delta_s}{\bar\delta_s+2D^2\delta_0},\frac{\delta_u}{\bar\delta_u+2D^2\delta_0}\right\}.
\]
Consequently, \(H\) and \(H^{-1}\) are fiberwise H\"older on bounded subsets of the fibers.
\end{theorem}

\begin{proof}
First note that the hypotheses of \cref{thm:HG} are satisfied. Indeed,
a strong exponential dichotomy with data \[
  (P,D,\delta_s,\bar\delta_s,\delta_u,\bar\delta_u)
\]
is in particular an ordinary exponential dichotomy with data \[
  (P,D,\delta), \qquad \delta:=\min\{\delta_s,\delta_u\}.
\]
Moreover, since \(\alpha<\alpha_0\), we have \[
  \delta_s-\alpha\bar\delta_s>0, \qquad \delta_u-\alpha\bar\delta_u>0.
\]
Combining \cref{eq:HoeldHGThm1} with\[
  \min\{\delta_s-\alpha\bar\delta_s,\delta_u-\alpha\bar\delta_u\} \leq \min\{\delta_s,\delta_u\} = \delta
\]
and, because \(D\geq1\), \[
  4D^{1+3\alpha}\geq 2D,
\]
we obtain \[
  \delta_0<\frac{\delta}{2D}
.\]
Thus, the smallness assumption required in \cref{thm:HG} is satisfied and we obtain a fixed point \(\eta^*\).

We first prove the H\"older regularity of \(\eta^*\). For \(R>0\), define \[
  \mathcal B_R^\alpha:=\{\xi\in\mathcal B:[\xi]_\alpha\le R\}.
\]
This set is nonempty, since \(0\in\mathcal B_R^\alpha\), and it is closed in \(\mathcal B\). Indeed, if \(\xi_n\in\mathcal B_R^\alpha\) and \(\xi_n\to\xi\) in \(\mathcal B\), then for \(v,w\in E_b\), \(v\neq w\), \[
  \frac{\|\xi(v)-\xi(w)\|}{\|v-w\|^\alpha}=\lim_{n\to\infty}\frac{\|\xi_n(v)-\xi_n(w)\|}{\|v-w\|^\alpha}\le R.
\]
Hence, \(\xi\in\mathcal B_R^\alpha\).

We show that \(F\) leaves such a closed H\"older ball invariant for a suitable \(R\). Let \(\xi\in\mathcal B_R^\alpha\), and put \[
  H_\xi:=\id+\xi.
\]
For \(x,y\) in the same fiber and \(r:=\|x-y\|\), we have, if \(r\le1\), \[
  \|H_\xi(x)-H_\xi(y)\|\le r+Rr^\alpha\le (1+R)r^\alpha.
\]
If \(r\geq1\), then by \cref{def:CarPert} we obtain \[
  \|\mathcal{R}_b(t,H_\xi(x))-\mathcal{R}_b(t,H_\xi(y))\|\le\delta_0\le\delta_0 r^\alpha\le\delta_0(1+R)r^\alpha.
\]
Thus, for almost every \(t\), \[
  \|\mathcal{R}_b(t,H_\xi(x))-\mathcal{R}_b(t,H_\xi(y))\|\le \delta_0(1+R)\|x-y\|^\alpha .
\]
Now fix \(v,w\in E_b\), and set \[
  z:=v-w.
\]
Using \cref{def:fixed-point-operator,eq:ED-stable}, and the previous bound with \(x=\Phi_{-\sigma}v\), \(y=\Phi_{-\sigma}w\), we get \[
  I_s\leq D\delta_0(1+R)\int_0^\infty e^{-\delta_s\sigma}\|\Phi_{-\sigma}z\|^\alpha\,d\sigma,
\]
where \(I_s\) denotes the contribution of the stable branch to \[
\|(F\xi)(v)-(F\xi)(w)\|
.\] 
Similarly, \[
  I_u\leq D\delta_0(1+R)\int_0^\infty e^{-\delta_u\sigma}\|\Phi_{\sigma}z\|^\alpha\,d\sigma
.\]
By \cref{lem:growth-SED} we estimate \[
  I_s\le D\delta_0(1+R)\left[\frac{(D^2)^\alpha}{\delta_s-\alpha\bar\delta_s}\|Pz\|^\alpha+\frac{D^\alpha}{\delta_s+\alpha\delta_u}\|Qz\|^\alpha\right],
\]
and \[
  I_u\le D\delta_0(1+R)\left[\frac{D^\alpha}{\delta_u+\alpha\delta_s}\|Pz\|^\alpha+\frac{(D^2)^\alpha}{\delta_u-\alpha\bar\delta_u}\|Qz\|^\alpha\right].
\]
Since \(\|Pz\|,\|Qz\|\leq D\|z\|\), we obtain \[
  \|(F\xi)(v)-(F\xi)(w)\|\le K_\alpha(1+R)\|v-w\|^\alpha,
\]
with \(K_\alpha\) as in the statement. Hence,
\[
  [F\xi]_\alpha\le K_\alpha(1+R)
.\]

Since \cref{eq:HoeldHGThm1} implies \(K_\alpha<1\), choose \[
  R:=\frac{K_\alpha}{1-K_\alpha}.
\]
Then \[
  K_\alpha(1+R)=R,
\]
and hence \[
  F(\mathcal B_R^\alpha)\subseteq\mathcal B_R^\alpha.
\]
Starting the Banach iteration at \(0\), all iterates remain in \(\mathcal B_R^\alpha\), and they converge in \(\mathcal B\) to the unique fixed point \(\eta^*\). Since \(\mathcal B_R^\alpha\) is closed in \(\mathcal B\), we have \[
  \eta^*\in\mathcal B_R^\alpha.
\]
Thus, \[
  [\eta^*]_\alpha\le R=\frac{K_\alpha}{1-K_\alpha}.
\]

It remains to prove the H\"older estimate for the inverse correction. Recall that \[
  \zeta(v)=-\int_0^\infty \Phi_\sigma P\,\mathcal{R}_b(-\sigma,\Psi_{-\sigma}(v))\,d\sigma+\int_0^\infty \Phi_{-\sigma}Q\,\mathcal{R}_b(\sigma,\Psi_\sigma(v))\,d\sigma .
\]
We first record growth estimates for the nonlinear flow. Put \[
  M:=2D^2.
\]
For \(t\geq0\), the strong dichotomy estimates \cref{lem:growth-SED} imply \[
  \|\Phi_t u\|\le M e^{\bar\delta_u t}\|u\|,\qquad \|\Phi_{-t}u\|\le M e^{\bar\delta_s t}\|u\|.
\]
Indeed, decompose \(u=Pu+Qu\); the stable part contracts forward and grows backward at rate at most \(\bar\delta_s\), while the unstable part contracts backward and grows forward at rate at most \(\bar\delta_u\).

Let \(v,w\in E_b\). For \(t\geq0\), the variation-of-constants formula and the Lipschitz bound give \[
  \|\Psi_t(v)-\Psi_t(w)\|\le M e^{\bar\delta_u t}\|v-w\|+\int_0^t M e^{\bar\delta_u(t-s)}\delta_0\|\Psi_s(v)-\Psi_s(w)\|\,ds.
\]
Multiplying by \(e^{-\bar\delta_u t}\) and applying Gronwall gives \begin{equation}\label{eq:HoeldHG1}
  \|\Psi_t(v)-\Psi_t(w)\|\le M e^{(\bar\delta_u+M\delta_0)t}\|v-w\|,\qquad t\geq0.
\end{equation}
Similarly, using the variation-of-constants formula for negative time, we obtain \begin{equation}\label{eq:HoeldHG2}
  \|\Psi_{-t}(v)-\Psi_{-t}(w)\|\le M e^{(\bar\delta_s+M\delta_0)t}\|v-w\|,\qquad t\geq0.
\end{equation}

Let \[
  0<\beta<\min\left\{\frac{\delta_s}{\bar\delta_s+M\delta_0},\frac{\delta_u}{\bar\delta_u+M\delta_0}\right\}.
\]
Since \[
  \min\{1,r\}\le r^\beta\qquad(r\geq0),
\]
the Lipschitz assumption on \(\mathcal{R}\) gives, for almost every \(t\), \[
  \|\mathcal{R}_b(t,x)-\mathcal{R}_b(t,y)\|\le \delta_0\|x-y\|^\beta.
\]
Therefore, the stable branch of \(\zeta(v)-\zeta(w)\) is bounded by \[
  D\delta_0\int_0^\infty e^{-\delta_s\sigma}\|\Psi_{-\sigma}(v)-\Psi_{-\sigma}(w)\|^\beta\,d\sigma.
\]
Using \eqref{eq:HoeldHG2} this is at most \[
  D\delta_0M^\beta\int_0^\infty e^{-\delta_s\sigma}e^{\beta(\bar\delta_s+M\delta_0)\sigma}\,d\sigma\,\|v-w\|^\beta =\frac{D\delta_0M^\beta}{\delta_s-\beta(\bar\delta_s+M\delta_0)}\|v-w\|^\beta.
\]
Similarly, using \eqref{eq:HoeldHG1}, the unstable branch is bounded by \[
  \frac{D\delta_0M^\beta}{\delta_u-\beta(\bar\delta_u+M\delta_0)}\|v-w\|^\beta.
\]
Since \(M=2D^2\), we obtain \[
  \|\zeta(v)-\zeta(w)\|\le C_\beta^-\|v-w\|^\beta,
\]
where \[
  C_\beta^-:=D\delta_0(2D^2)^\beta\left[\frac{1}{\delta_s-\beta(\bar\delta_s+2D^2\delta_0)}+\frac{1}{\delta_u-\beta(\bar\delta_u+2D^2\delta_0)}\right].
\]
The denominators are positive precisely for \[
  \beta<\min\left\{\frac{\delta_s}{\bar\delta_s+2D^2\delta_0},\frac{\delta_u}{\bar\delta_u+2D^2\delta_0}\right\}.
\]
Hence, \(\zeta\) is fiberwise \(\beta\)-H\"older for every such \(\beta\).

Finally, since \(H=\id+\eta^*\) and \(H^{-1}=\id+\zeta\), the correction terms are globally fiberwise H\"older. On a bounded subset of a fiber, the identity map is also \(\alpha\)-H\"older, respectively \(\beta\)-H\"older, after adjusting the constant. Thus, \(H\) and \(H^{-1}\) are fiberwise H\"older on bounded fiber subsets.
\end{proof}

\begin{remark}\label{rem:BP-rigidity}
   For control systems, linearization by static feedback, i.e. by transformations acting pointwise on the state and the current control value, is rare: at generic points, topological linearizability of this kind already implies smooth linearizability \cite{BaratchartPomet2009LocalLinearization}, meaning topological conjugacy to the linear approximation is ``as rare as differential conjugacy''. This is contrasted by the situation in our setting, where topological and Hölder-linearizations depend only on spectral and smallness conditions, whereas smooth linearizations additionally require non-resonance conditions, as discussed following \cref{lem:growth-SED}. An approach to obtain topological conjugacy results for control-affine systems, where control is instead viewed as a base, will be discuss in \cref{subsec:control-affine}.
\end{remark}

\subsection{Local Hartman--Grobman theorem}
\label{subsec:local-HG}

For \(r>0, b \in B\) write \[
  \mathcal{E}(r):=\{v\in \mathcal{E}:\|v\|<r\},\qquad E_b(r):=E_b\cap \mathcal{E}(r).
\]
\begin{definition}[Local Carath\'eodory perturbation]
Let \(r_0>0\). A family \[
  \widetilde{\mathcal{R}}_b(t,\cdot):E_{b\cdot t}(r_0)\to E_{b\cdot t}
\]
is called a local Carath\'eodory perturbation if it satisfies the measurability, coherence, and zero-section assumptions of \cref{def:CarPert} whenever the involved curves remain in \(\mathcal{E}(r_0)\).
\end{definition}

We now record the local form of the theorem. The argument is the usual localization by a fiberwise cutoff: one replaces the local perturbation by a globally bounded Carath\'eodory perturbation, applies the global theorem, and then restricts to the region where the cutoff is identically one.

\begin{lemma}[Cutoff estimate]\label{lem:cutoff-estimate}
Let \(\widetilde{\mathcal{R}}\) be defined on \(\mathcal{E}(r_0)\), vanish on the zero section, and satisfy \[
  L(r):=\sup_b \operatorname*{ess\,sup}_{t\in\R} \mathrm{Lip}\bigl(\widetilde{\mathcal{R}}_b(t,\cdot)|_{E_{b\cdot t}(2r)} \bigr)\to 0 \qquad (r \to 0).
\]
For \(0<2r<r_0\), let \[
  \mathcal{R} _b^r(t,v) :=
  \begin{cases}
  \chi(\|v\|/r)\widetilde{\mathcal{R}}_b(t,v),& \|v\|<2r,\\
  0,& \|v\|\ge 2r,
  \end{cases}
\]
where \(\chi\colon[0,\infty)\to[0,1]\) is Lipschitz with \(\chi=1\) on \([0,1]\) and \(\chi=0\) on \([2,\infty)\). Then \(\mathcal{R}^r\) is a global Carath\'eodory perturbation, \(\mathcal{R}^r=\widetilde{\mathcal{R}}\) on \(\mathcal{E}(r)\), and \[
  \|\mathcal{R}_b^r(t,v)-\mathcal{R}_b^r(t,w)\| \leq \delta_r\min\{1,\|v-w\|\}, \qquad
\] with \(  \delta_r\to 0\) as \( r \to 0\).
\end{lemma}

\begin{proof}
Measurability and coherence are inherited from \(\widetilde{\mathcal{R}}\),
since the cutoff only depends on the fiber norm. Moreover, on \(\mathcal{E}(2r)\), \[
  \|\widetilde{\mathcal{R}}_b(t,v)\| \leq L(r)\|v\| \leq 2rL(r).
\]
Using the Lipschitz continuity of \(\chi\) and of the norm on each fiber,
one obtains \[
  \operatorname{Lip}(\mathcal{R}_b^r(t,\cdot)) \leq C_\chi L(r)
\]
for a constant \(C_\chi\) depending only on \(\chi\). Also, since \(0\le\chi\le1\), \[
  \|\mathcal{R}_b^r(t,v)\|\leq 2rL(r).
\]
Hence, we get \[
    \|\mathcal{R}^{r}_b(t,v)- \mathcal{R}^{r}_b(t,w)\| \leq \|\mathcal{R}^{r}_b(t,v)\| + \|\mathcal{R}^{r}_b(t,w)\| \leq 4r L(r)   
,\] which means \[
  \|\mathcal{R}_b^r(t,v)-\mathcal{R}_b^r(t,w)\| \leq \delta_r\min\{1,\|v-w\|\}, 
\]
where \( \delta_r:=\max\{C_\chi L(r),4rL(r)\}\), so since \(L(r)\to 0\), also \(\delta_r\to 0\).
\end{proof}

\begin{theorem}[Local Hartman--Grobman theorem]\label{thm:local-HG}
Let \(\Phi\) have an ED with data \((P,D,\delta)\) and \(\widetilde{\mathcal R}\) be a local Carath\'eodory perturbation on \(\mathcal{E}(r_0)\), with \(\widetilde{\mathcal R}_b(t,0)=0\), and assume \[
  L(r):=\sup_b\operatorname*{ess\,sup}_{t\in\mathbb R} \operatorname{Lip}\bigl( \widetilde{\mathcal R}_b(t,\cdot)|_{E_{b\cdot t}(2r)} \bigr)\to0
  \qquad(r\to0).
\]
Denote by \(\widetilde\Psi\) the corresponding local flow.

For \(0<2r<r_0\), let \(\mathcal R^r\) be the standard fiberwise cutoff of
\(\widetilde{\mathcal R}\). Then \(\mathcal R^r\) is a global
Carath\'eodory perturbation with constant \(\delta_r\to0\), and hence,
by \cref{prop:cutoff-flow}, generates a unique global solution cocycle
\(\Psi^r\).

For all sufficiently small \(r>0\), there is a fiberwise homeomorphism \[
  H_r=\id+\eta_r
\]
such that \[
  H_r(\Phi_t v)=\Psi^r_t(H_r(v))
  \qquad(t\in\mathbb R,\ v\in \mathcal{E}).
\]
Moreover, \[
  \|\eta_r\|_{\mathcal B} \leq \frac{4DrL(r)}{\delta},
  \qquad\text{so}\qquad \frac{\|\eta_r\|_{\mathcal B}}{r}\to0 .
\]
In particular, for \(r>0\) small enough,\[ 
  \rho_r:=r-\|\eta_r\|_{\mathcal B}>0.
\]

If \[
  \Phi_s v\in \mathcal{E}(\rho_r)
\]
for all \( s\) between \( 0\) and \( t\), then \[
  H_r(\Phi_t v)=\widetilde\Psi_t(H_r(v)).
\]
Thus, \(H_r\) conjugates the linear flow to the original local flow along
such orbit segments.

If the cutoff system satisfies \[
  F^r(\mathcal C)\subseteq\mathcal C, \qquad \zeta^r\in\mathcal C,
\]
then \(H_r\) and \(H_r^{-1}\) are continuous bundle maps on their local
domains and \( \Psi^{r}\) is a continuous SPF. If \(\Phi\) has a strong ED, the corresponding local conjugacy is fiberwise H\"older by \cref{thm:Holder-HG}.
\end{theorem}

\begin{proof}
Choose a Lipschitz cutoff \(\chi\colon[0,\infty)\to[0,1]\) with \(\chi=1\) on \([0,1]\) and
\(\chi=0\) on \([2,\infty)\), and set \[
  \mathcal R_b^r(t,v) = \chi(\|v\|/r)\widetilde{\mathcal R}_b(t,v)
\]
on \(\mathcal{E}(2r)\), extended by \(0\) outside \(\mathcal{E}(2r)\). Since
\(\widetilde{\mathcal R}_b(t,0)=0\), one has \[
  \|\widetilde{\mathcal R}_b(t,v)\|\le 2rL(r) \qquad(\|v\|<2r),
\]
and the cutoff estimate gives \[
  \|\mathcal R_b^r(t,v)-\mathcal R_b^r(t,w)\| \leq \delta_r\min\{1,\|v-w\|\},
\]
and \( \delta_r\to0 \) as \( r \to 0\). Moreover, \(\mathcal R^r=\widetilde{\mathcal R}\) on \(\mathcal{E}(r)\).

Choose \(r\) small enough such that \[
  \delta_r<\frac{\delta}{2D}.
\]
By \cref{prop:cutoff-flow}, \(\mathcal R^r\) generates a unique global
solution cocycle \(\Psi^r\). Applying \cref{thm:HG} to the cutoff system
gives \(H_r=\id+\eta_r\) and \[
  H_r(\Phi_t v)=\Psi^r_t(H_r(v)).
\]
Since \(\eta_r\) is the fixed point of the operator \(F\) of \cref{def:fixed-point-operator} built with \(\mathcal R^r\), and \(\|\mathcal R^r_b(t,v)\|\le 2rL(r)\) for all \(b\), \(v\) and almost all \(t\) by the proof of \cref{lem:cutoff-estimate}, the estimate in the proof of \cref{lem:F-welldef}, with \(\delta_0\) replaced by this bound, gives \[
  \|\eta_r\|_{\mathcal B}=\|F\eta_r\|_{\mathcal B}\le\frac{2D}{\delta}\,2rL(r)=\frac{4DrL(r)}{\delta}.
\]
Since \(L(r)\to0\), we get \(\|\eta_r\|_{\mathcal B}<r\) for \(r\) small, that is, \[
  \rho_r:=r-\|\eta_r\|_{\mathcal B}>0.
\]

Now assume \(\Phi_s v\in \mathcal{E}(\rho_r)\) for all \(s\) between \(0\) and
\(t\). Using the established conjugation \( \Psi^{r}_s(H_r(v)) = H_r(\Phi_s v)\), we calculate \[
  \| \Psi^{r}_s(H_r(v))\| =  \|H_r(\Phi_s v)\| \leq \|\Phi_s v\|+\|\eta_r\|_{\mathcal B} <r
,\] 
meaning this orbit also solves the original local equation. By \cref{prop:cutoff-flow}, \[
  \Psi_s^r(H_r(v))=\widetilde\Psi_s(H_r(v))
\]
on the considered interval, and in particular \[
  H_r(\Phi_t v)=\widetilde\Psi_t(H_r(v))
.\]

The continuity and H\"older assertions follow by applying the
corresponding parts of \cref{thm:HG} and \cref{thm:Holder-HG} to the
cutoff system.
\end{proof}

\subsection{Nonuniform Hartman--Grobman theorems}
\label{subsec:nonuniform-HG}

We now record the nonuniform versions. The point is that there is no new fixed-point argument: after passing to Lyapunov norms, the nonuniform dichotomy becomes a uniform dichotomy on an adapted Banach bundle, and the preceding theorems apply verbatim. We denote in this section the spaces of bundle maps in the adapted Lyapunov norms of \cref{thm:ED-NUED-equiv} and \cref{thm:StNUED-StED-Equiv} by \( \mathcal{B}^{\mathrm L}\) (resp.\ \( \mathcal{C}^{\mathrm L}\) and the norm \( \|\cdot \|^{\mathrm L} \)).

\begin{theorem}[Nonuniform Hartman--Grobman theorem]
\label{thm:NU-HG-continuous}
Assume that \(\Phi\) has a nonuniform exponential dichotomy with data \[
  (P,D,\varepsilon,\delta).
\]
Fix \[
  0<\delta'<\delta,
\]
and let \(\|\cdot\|^{\mathrm L}\) denote the Lyapunov bundle norm associated with \(\delta'\). Assume that the Carath\'eodory source \(\mathcal{R}\) satisfies the hypotheses of \cref{def:CarPert} with respect to \(\|\cdot\|^{\mathrm L}\), with Lipschitz constant \(\delta_0\), and that \[
  \delta_0<\frac{\delta'}{2}.
\]
Assume also the continuity-layer condition in the Lyapunov bundle, \[
  F^{\mathrm L}(\mathcal C^{\mathrm L})\subseteq \mathcal C^{\mathrm L}, \qquad \zeta^{\mathrm L}\in\mathcal C^{\mathrm L},
\]
where \(F^{\mathrm L}\) denotes the fixed-point operator computed using the Lyapunov norm, and \(\zeta^{\mathrm L}\) is the inverse correction. Then there is a unique \(\eta^*\in\mathcal B^{\mathrm L}\) such that \(H:=\id+\eta^*\) satisfies \[
  H(\Phi_t v)=\Psi_t(H(v))\qquad(t\in\R,\ v\in \mathcal{E}).
\]
Moreover, \[
  \eta^*,\zeta^{\mathrm L}\in\mathcal C^{\mathrm L},
\]
and \(H\) restricts to a homeomorphism on every fiber, with inverse \[
  H^{-1}=\id+\zeta^{\mathrm L}.
\]
The maps \(H\) and \(H^{-1}\) are jointly continuous for the adapted Lyapunov bundle topology, and hence also for the original bundle topology.
\end{theorem}

\begin{proof}
By the Lyapunov-norm reduction for NUED, \cref{thm:ED-NUED-equiv}, the same linear SPF \(\Phi\), viewed
on the Lyapunov bundle \((\mathcal{E},\pi,B,\|\cdot\|^{\mathrm L})\), has a uniform ED
with data
\[
  (P,1,\delta').
\]
Thus, the smallness condition becomes
\[
  \delta_0<\frac{\delta'}{2}.
\]
The normal Hartman--Grobman theorem applied on the Lyapunov bundle gives the
unique fixed point \(\eta^*\in\mathcal B^{\mathrm L}\), the conjugacy identity,
and the fiberwise inverse \(H^{-1}=\id+\zeta^{\mathrm L}\). The additional
continuity-layer hypothesis gives
\[
  \eta^*,\zeta^{\mathrm L}\in\mathcal C^{\mathrm L}.
\]
Finally, the Lyapunov norm is locally equivalent to the original fiber norm,
so the adapted and original Banach bundle topologies agree. Hence, joint
continuity in the Lyapunov bundle is joint continuity in the original bundle.
\end{proof}

\begin{theorem}[Nonuniform H\"older Hartman--Grobman theorem]
\label{thm:NU-HG-Holder}
Assume that \(\Phi\) has a strong nonuniform exponential dichotomy with data
\[
  (P,D,\varepsilon,\delta_s,\bar\delta_s,\delta_u,\bar\delta_u).
\]
Choose rates
\[
  0<\delta_s'<\delta_s,\qquad 0<\delta_u'<\delta_u,
\]
and
\[
  \bar\delta_s'>\bar\delta_s+\varepsilon,\qquad
  \bar\delta_u'>\bar\delta_u+\varepsilon.
\]
Let \(\|\cdot\|^{\mathrm{str}}\) be the corresponding strong Lyapunov norm.
Set
\[
  \alpha_0':=\min\left\{\frac{\delta_s'}{\bar\delta_s'},\frac{\delta_u'}{\bar\delta_u'}\right\}.
\]
Assume that \(\mathcal{R}\) satisfies the hypotheses of \cref{def:CarPert} with
respect to \(\|\cdot\|^{\mathrm{str}}\), with Lipschitz constant \(\delta_0\).
Let \(\alpha\in(0,\alpha_0')\), and assume
\[
  \delta_0<\frac{\min\{\delta_s'-\alpha\bar\delta_s',\delta_u'-\alpha\bar\delta_u'\}}{4}.
\]
Then the correction term \(\eta^*\) of the nonuniform Hartman--Grobman
conjugacy is fiberwise \(\alpha\)-H\"older in the strong Lyapunov norm:
\[
  \|\eta^*(v)-\eta^*(w)\|^{\mathrm{str}}\le
  \frac{K_\alpha^{\mathrm{str}}}{1-K_\alpha^{\mathrm{str}}}
  \bigl(\|v-w\|^{\mathrm{str}}\bigr)^\alpha
\]
for all \(v,w\) in the same fiber, where
\[
  K_\alpha^{\mathrm{str}}
  :=
  \delta_0\left[
  \frac{1}{\delta_s'-\alpha\bar\delta_s'}
  +\frac{1}{\delta_u'-\alpha\bar\delta_u'}
  +\frac{1}{\delta_s'+\alpha\delta_u'}
  +\frac{1}{\delta_u'+\alpha\delta_s'}
  \right].
\]
Moreover, if \(H^{-1}=\id+\zeta\), then \(\zeta\) is fiberwise
\(\beta\)-H\"older in the strong Lyapunov norm for every
\[
  \beta<
  \min\left\{
  \frac{\delta_s'}{\bar\delta_s'+2\delta_0},
  \frac{\delta_u'}{\bar\delta_u'+2\delta_0}
  \right\}.
\]
Consequently, \(H\) and \(H^{-1}\) are fiberwise H\"older on bounded subsets
of the fibers, measured in the strong Lyapunov norm.
\end{theorem}

\begin{proof}
By \cref{thm:StNUED-StED-Equiv}, the same linear SPF, viewed in the
strong Lyapunov norm, has a strong ED with data
\[
  (P,1,\delta_s',\bar\delta_s',\delta_u',\bar\delta_u').
\]
Thus, \cref{thm:Holder-HG} applies with \(D=1\). Therefore, its smallness
condition becomes
\[
  \delta_0< \frac{\min\{\delta_s'-\alpha\bar\delta_s', \delta_u'-\alpha\bar\delta_u'\}}{4},
\]
and the constant \(K_\alpha\) reduces to
\[
  K_\alpha^{\mathrm{str}}
  =
  \delta_0\left[
  \frac{1}{\delta_s'-\alpha\bar\delta_s'}
  +\frac{1}{\delta_u'-\alpha\bar\delta_u'}
  +\frac{1}{\delta_s'+\alpha\delta_u'}
  +\frac{1}{\delta_u'+\alpha\delta_s'}
  \right].
\]
\end{proof}

\begin{remark}[Original-norm sufficient condition]
\label{rem:NU-weighted-smallness}
Since the Lyapunov norms of \cref{thm:ED-NUED-equiv} and \cref{thm:StNUED-StED-Equiv} satisfy
\[
  \|v\| \leq \|v\|^{\mathrm L}\le 2D N(b)\|v\|,\qquad v\in E_b,
\] and \[
    \|v\| \leq \|v\| ^{\mathrm{L}} \leq 2 D^{2}N(b)^{2}\|v\|, \qquad v \in E_b 
\] 
respectively, a sufficient way to obtain the Lyapunov-norm Lipschitz bounds is the weighted original-norm estimate
\[
  \|\mathcal{R}_b(t,v)-\mathcal{R}_b(t,w)\| \leq \frac{\delta_0}{2D N(b\cdot t)}\min\{1,\|v-w\|\}, \qquad v,w\in E_{b\cdot t}
\] for the standard case and \[
     \|\mathcal{R}_b(t,v)-\mathcal{R}_b(t,w)\| \leq \frac{\delta_0}{2D^{2} N(b\cdot t)^{2}}\min\{1,\|v-w\|\}, \qquad v,w\in E_{b\cdot t}
\] for the Hölder case. 
Indeed, multiplying by the local distortion factor \(2D N(b\cdot t)\) (resp.\ \( 2D^{2}N(b\cdot t)^{2}\)) gives the desired estimate in \(\|\cdot\|^{\mathrm L}\).
\end{remark}

\begin{remark}[Endpoint loss]
\label{rem:NU-Holder-endpoint-loss}
Since the primed rates can be chosen arbitrarily close to
\[
  \delta_s,\qquad \delta_u,\qquad \bar\delta_s+\varepsilon,\qquad \bar\delta_u+\varepsilon,
\]
the limiting forward H\"older range is, up to endpoint loss,
\[
  \alpha< \min \left\{\frac{\delta_s}{\bar\delta_s+\varepsilon}, \frac{\delta_u}{\bar\delta_u+\varepsilon}\right\}.
\]
Similarly, the limiting inverse range is, again up to endpoint loss,
\[
  \beta< \min\left\{\frac{\delta_s}{\bar\delta_s+\varepsilon+2\delta_0},\frac{\delta_u}{\bar\delta_u+\varepsilon+2\delta_0}\right\}.
\]
All constants here are measured in the strong Lyapunov norm.
\end{remark}

\begin{remark}
  By the norm comparison in \cref{thm:StNUED-StED-Equiv}, the
\(\alpha\)-Hölder estimate in the strong Lyapunov norm of \cref{thm:NU-HG-Holder} implies, on each
fiber \(E_b\),
\[
  \|\eta^*(v)-\eta^*(w)\| \leq \frac{K_\alpha^{\mathrm{str}}}{1-K_\alpha^{\mathrm{str}}}\bigl(2D^2N(b)^2\bigr)^\alpha \|v-w\|^\alpha .
\]
Thus, the Hölder exponent is unchanged in the original fiber norm, but the
Hölder constant becomes nonuniform in \(b\). If \(N\) is bounded on a
base subset, the constants are uniform on that subset.
\end{remark}
\section{Applications}\label{sec:App}
\subsection{Control-affine systems on manifolds}\label{subsec:control-affine}
For control systems, the main Hartman--Grobman-type theorem is an abstract theorem by Baratchart, Chyba, and Pomet~\cite[Thm.~2.1]{BaratchartEtAl2007GrobmanHartmana} which allows for a \emph{constant} linear part, given by a finite-dimensional matrix \( A\) and a compatible non-autonomous non-linearity. Our results extend these in several directions. Not only do we allow for linearization along arbitrary solutions, since we allow for the non-autonomy of the linear part, but we can also strengthen the result to manifolds and their inherent non-trivial bundle structure. Furthermore, local and Hölder versions can be obtained in this setting; both are stated in \cref{thm:HG-control}. 

Let \(M\) be a connected \(C^k\)-manifold, \(k \geq 3\), equipped with a \(C^k\)-Riemannian metric. Consider \[
  \dot x=F_0(x)+\sum_{j=1}^m u_j(t)F_j(x),
\]
where \(F_0,\ldots,F_m\in \Gamma^k(TM)\). Let \(U\subset\R^m\) be compact and convex, and set \[
  \mathcal U:=\{u\in L^\infty(\R,\R^m):u(t)\in U\text{ for a.e.\ }t\}.
\]
We equip \(\mathcal U\) with the weak-* topology and write \[
  \theta_tu:=u(t+\cdot).
\]
By Kawan's weak-* continuity results, \(\mathcal U\) is compact metrizable, the shift \(\theta:\R\times\mathcal U\to\mathcal U\) is continuous, and the associated control flow \[
  \Psi_t(u,x):=(\theta_tu,\psi(t,u,x))
\]
is continuous, provided the solutions are globally defined; see~\cite[Props.~1.14, 1.15, 1.17]{Kawan2013InvarianceEntropy}. Moreover, since the vector fields are \(C^k\), the derivative cocycle depends continuously on \((t,u,x)\); see~\cite[Thm.~1.1]{Kawan2013InvarianceEntropy}.

Assume that there is a continuous invariant equilibrium section \[
  \sigma:\mathcal U\to M,\qquad \psi(t,u,\sigma(u))=\sigma(\theta_tu).
\]
We write \[
  \mathcal{E}_\sigma:=\sigma^*TM=\{(u,v):u\in\mathcal U,\ v\in T_{\sigma(u)}M\}
\]
for the pulled-back tangent bundle over \(\mathcal U\).

The intrinsic linearization along \(\sigma\) is the derivative cocycle \[
  \Phi(t,u):=d_x\psi(t,u,\sigma(u)):T_{\sigma(u)}M\to T_{\sigma(\theta_tu)}M.
\]
Equivalently, using the Levi-Civita connection, \(\Phi(t,u)\) is the solution operator of the covariant variational equation along \(p(t):=\sigma(\theta_tu)\), \[
  \nabla_t y(t)=\left(\nabla F_0(p(t))+\sum_{j=1}^m u_j(t)\nabla F_j(p(t))\right)y(t).
\]
Thus, \(\Phi\) is a continuous linear SPF on \(\mathcal{E}_\sigma\) over the shift.

From here on out we will always assume that our manifold has an injectivity radius bounded away from \( 0\) on \( \sigma(\mathcal U)\). This is always the case if \( M\) is compact.

\begin{lemma}[Tubular Carath\'eodory representation]
\label{lem:control-tubular}
Given our assumptions on the control system, there exists \(r_0>0\) such that the fiberwise exponential chart \[
  \mathcal{E}_\sigma(r_0)\ni(u,v)\longmapsto (u,\exp_{\sigma(u)}v)\in\mathcal U\times M
\]
is a homeomorphism onto a neighborhood of \(\operatorname*{graph}\sigma\). In this chart the control flow is represented by
\[
  \widetilde\Psi_t(u,v):=\exp_{\sigma(\theta_tu)}^{-1}\psi(t,u,\exp_{\sigma(u)}v).
\]
The map \(\widetilde\Psi\) is a continuous local skew-product flow on
\(\mathcal{E}_\sigma(r_0)\), and it satisfies \[
  \widetilde\Psi_t(u,v)=\Phi(t,u)v+\int_0^t\Phi(t-s,\theta_su)\widetilde{\mathcal R}_u(s,\widetilde\Psi_s(u,v))\,ds.
\]
Moreover, \[
  \widetilde{\mathcal R}_u(t,w)=R_0(\theta_tu,w)+\sum_{j=1}^m u_j(t)R_j(\theta_tu,w)
\]
for a.e.\ \(t\), where the coefficients \(R_i\) are jointly continuous on the local bundle, fiberwise \(C^{k-1}\), and satisfy \[
  R_i(u,0)=0,\qquad d_wR_i(u,0)=0.
\]
In particular, \[
  L(r):=\sup_i\sup_{u\in\mathcal U}\mathrm{Lip}\bigl(R_i(u,\cdot)|_{(\mathcal{E}_\sigma)_u(2r)}\bigr)\to0\qquad(r\to0).
\]
The source \(\widetilde{\mathcal R}\) satisfies the Carath\'eodory measurability and coherence assumptions of \cref{def:CarPert}.
\end{lemma}

\begin{proof}
Continuity of \(\widetilde\Psi\) follows directly from Kawan's continuity theorem for the control flow, the continuity of the shift, the continuity of \(\sigma\), and the continuity of the exponential and logarithm maps on the chosen tube.

It remains to identify the local equation. Put \[
  p(t):=\sigma(\theta_tu),\qquad q(t):=\psi(t,u,\exp_{\sigma(u)}v),\qquad w(t):=\exp_{p(t)}^{-1}q(t).
\]
The curves \(p\) and \(q\) are controlled trajectories. Differentiating \(w(t)=\exp_{p(t)}^{-1}q(t)\) by the Levi-Civita connection gives, for a.e.\ \(t\),  \[
  \nabla_t w(t)=G_0(\theta_tu,w(t))+\sum_{j=1}^m u_j(t)G_j(\theta_tu,w(t)).
\]
The coefficients \(G_j\) are obtained from the vertical differentials of the logarithm map in the two variables. 

At \(w=0\), the identities \(D_1^v\log(y,y)=-\id\) and \(D_2^v\log(y,y)=\id\) give \(G_j(u,0)=0\). In normal coordinates at \(\sigma(u)\), \[
  G_j(u,w)=\nabla F_j(\sigma(u))w+O(\|w\|^2),
\]
uniformly on the chosen tube. Thus, setting \[
  A_j(u):=\nabla F_j(\sigma(u)),\qquad R_j(u,w):=G_j(u,w)-A_j(u)w,
\]
gives \(R_j(u,0)=0\), \(d_wR_j(u,0)=0\), and the asserted regularity. Integrating the resulting covariant equation against the linear cocycle \(\Phi\) yields the variation-of-constants formula.
Finally, coherence follows from \((\theta_su)_j(t)=u_j(t+s)\), and orbitwise measurability follows from the continuity of the coefficients \(R_j\) and the measurability of the scalar controls \(u_j\).
\end{proof}

\begin{remark}[No pointwise perturbation on the weak-* hull]
\label{rem:no-pointwise-control-source}
The source in \cref{lem:control-tubular} should not be collapsed to a pointwise bundle map \(f:\mathcal{E}_\sigma\to \mathcal{E}_\sigma\). Formally one would have to write \[
  f(u,w)=R_0(u,w)+\sum_{j=1}^m u_j(0)R_j(u,w),
\]
but \(u_j(0)\) is not well-defined on \(L^\infty\)-classes and is not weak-* continuous. Thus, the control application genuinely uses the Carath\'eodory source \[
  \widetilde{\mathcal R}_u(t,w)=R_0(\theta_tu,w)+\sum_{j=1}^m u_j(t)R_j(\theta_tu,w).
\]
\end{remark}

\begin{proposition}[Control-affine sources satisfy the continuity layer]
\label{prop:control-continuity-layer}
Let
\begin{equation}\label{eq:ConAffSource}
  \mathcal S_u(t,w)=S_0(\theta_tu,w)+\sum_{j=1}^m u_j(t)S_j(\theta_tu,w)
\end{equation}
be a control-affine Carath\'eodory source with jointly continuous coefficients
\(S_j\), which vanish on the zero section. Additionally, we assume that \( \mathcal{S}\) satisfies the smallness condition of \cref{thm:HG} and each coefficient \( S_j\) is bounded individually.

Assume that the corresponding nonlinear flow is continuous. Then
\[
  F(\mathcal C)\subseteq\mathcal C,\qquad \zeta\in\mathcal C.
\]
\end{proposition}

\begin{proof}
  Similarly to \cref{prop:ContForContBunMap} we need to show that if \( \eta \in \mathcal{C}\) and \( F\) is built with \( \mathcal{S}_u\), i.e. \begin{align*}
      (F\eta)(u,v) = &\int_{0}^{\infty} \Phi(t,\theta_{-t}u) P(\theta_{-t}u)\mathcal{S}_u(-t,H_{\eta}(\theta_{-t}u,\Phi(-t,u)v)) \, dt \\ 
      &- \int_{0}^{\infty} \Phi(-t,\theta_t u) Q(\theta_t u) \mathcal{S}_u(t,H_\eta(\theta_t u, \Phi(t,u)v)) \, dt,
  \end{align*} 
  then \( (F\eta) \in \mathcal{C}\). So fix \( (u_n,v_n)\to(u_0,v_0)\) in the pulled-back tangent bundle \( E_{\sigma}\).
  Similarly, again, to \cref{prop:ContForContBunMap} we focus on the stable component \( (F\eta)^{s}\). Since \( \mathcal{S}_{u_n}\) is given as a finite sum the integrands of \( (F\eta)^{s}\) are of the form \[
       \kappa_{n,0}(t) + \sum_{j=1}^{m} u_{n,j}(-t)\kappa_{n,j}(t)
  ,\] 
 where the functions \( \kappa_{n,j}\) are given by  
  \[
      \kappa_{n,j}(t) := \Phi(t,\theta_{-t}u_n)P(\theta_{-t}u_n)S_j(\theta_{-t}u_n,H_{\eta}(\theta_{-t}u_n,\Phi(-t,u_n)v_n)).
       \] 
 The pointwise convergence, for fixed \( t \geq 0\), \( \kappa_{n,j}(t) \to \kappa_{0,j}(t)\) follows then directly from the following:
  \begin{itemize}
    \item \( u_n \to u_0\) weak-* in \( \mathcal{U}\).
    \item the shift \( \theta_{-t}\) is continuous in the weak-* sense on \( \mathcal{U}\).
    \item \( \Phi\) is continuous.
    \item \( P\) is continuous along sections. 
    \item \( S_j\) is jointly continuous. 
    \item \( \id + \eta\) is jointly continuous.
  \end{itemize}
  
  To pass from this pointwise convergence of the integrands to convergence of the integral we are using again, as in \cref{prop:ContForContBunMap}, a dominated convergence argument. Now since we assumed individual boundedness of \( S_j\) define \[
      C_j := \sup_{u,w}\|S_j(u,w)\| < \infty  
  ,\] which, using \cref{eq:ED-stable}, gives us \[
      \|\kappa_{n,j}(t)\| \leq DC_j e^{-\delta t} 
  .\] Consequently, \[
      \|\kappa_{n,j}(t)-\kappa_{0,j}(t)\| \leq 2 D C_j e^{-\delta t} 
  ,\] hence the right-hand side is an integrable majorant on \( [0,\infty)\). Dominated convergence now guarantees that \( \kappa_{n,j} \to \kappa_{0,j}\) in \( L^{1}( \R_{+})\). For the \( S_0\)-term this suffices to show \[
      \int_{0}^{\infty } \kappa_{n,0}(t) \, dt \to \int_{0}^{\infty} \kappa_{0,0}(t) \, dt
  .\] 
  To control the other \( S_j\)-terms we understand \[
     \int_{0}^{\infty} u_{n,j}(-s) \kappa_{n,j}(s) \, ds - \int_{0}^{\infty} u_{0,j}(-s)\kappa_{0,j}(s) \, ds
  \] as the sum of \begin{equation} \label{eq:ConContInt1}
    \int_{0}^{\infty} u_{n,j}(-s) (\kappa_{n,j}(s)-\kappa_{0,j}(s)) \, ds
  \end{equation} and 
  \begin{equation}\label{eq:ConContInt2}
    \int_{0}^{\infty} (u_{n,j}(-s)- u_{0,j}(-s))\kappa_{0,j}(s) \, ds
  .\end{equation}
    Now \cref{eq:ConContInt1} tends to \( 0\) in norm, because the controls are uniformly \( L^{\infty}\) bounded and \cref{eq:ConContInt2} tends to \(0\) in norm, because \( \kappa_{0,j}\) is an \( L^{1}\) function and hence by the weak-* convergence of \( u_n\) in \( L^{\infty}\) this converges to \( 0\). This means now finally that the stable integrals converge \[
        (F\eta)^{s}(u_n,v_n) \to (F\eta)^{s}(u_0,v_0)
    \] and an analogous argumentation shows \[
        (F\eta)^{u}(u_n,v_n) \to (F\eta)^{u}(u_0,v_0)
    ,\] which means that \( F\eta\) is jointly continuous. 
    Again, as in \cref{prop:ContForContBunMap}, the statement for \( \zeta\) follows from an analogous argumentation, where the main difference is the replacement of the inner linear orbit \( \Phi_{\pm t}\) by the non-linear orbits \( \Psi_{\pm t}\), which we will omit for brevity's sake.
\end{proof}

\begin{theorem}[Hartman--Grobman theorem for control-affine systems]\label{thm:HG-control}
Assume that the control-affine system above admits a continuous invariant equilibrium section \(\sigma:\mathcal U\to M\). Let \[
  \Phi(t,u):T_{\sigma(u)}M\to T_{\sigma(\theta_tu)}M
\]
be the derivative cocycle along \(\sigma\), and assume that \(\Phi\) admits an exponential dichotomy on \(\mathcal{E}_\sigma\) with data \((P,D,\delta)\). Then, for \(r>0\) sufficiently small, there exists \( \rho >0\) such that \[
  V:=\{(u,v)\in \mathcal{E}_\sigma:\|v\|<\rho\}
\]
is a neighborhood of the zero section, a neighborhood \(W\subset\mathcal U\times M\) of \(\operatorname*{graph}\sigma\), and a bundle-respecting homeomorphism \[
  H:V\to W,\qquad H(u,0_u)=(u,\sigma(u)),
\]
such that \[
  H(\Phi_t(u,v))=\Psi_t(H(u,v))
\]
for every orbit segment for which \(\Phi_s(u,v)\in V\) for all \(s\) between \(0\) and \(t\). Moreover, \(H\) and \(H^{-1}\) are jointly continuous.

If \(\Phi\) has a strong exponential dichotomy and the non-linearity satisfies the stronger assumptions of \cref{thm:Holder-HG}, then the fiber maps of \(H\) and \(H^{-1}\), written in the exponential coordinates of \(\mathcal{E}_\sigma\), are H\"older with the exponents given by \cref{thm:Holder-HG}. If \(\Phi\) has a nonuniform exponential dichotomy, the same statements can be made in the respective adapted Banach bundle structures given in \cref{thm:ED-NUED-equiv} and \cref{thm:StNUED-StED-Equiv}.
\end{theorem}

\begin{proof}
By \cref{lem:control-tubular}, the control flow near \(\operatorname*{graph}\sigma\)
is represented on \(\mathcal{E}_\sigma\) by a continuous local bundle flow \(\widetilde\Psi\)
satisfying the Carath\'eodory variation-of-constants formula with source\[
  \widetilde{\mathcal R}_u(t,w)=R_0(\theta_tu,w)+\sum_{j=1}^m u_j(t)R_j(\theta_tu,w).
\]
Since the coefficients vanish to second order at the zero section,\[
  L(r):=\sup_i\sup_{u\in\mathcal U}\mathrm{Lip}\bigl(R_i(u,\cdot)|_{(\mathcal{E}_\sigma)_u(2r)}\bigr)\to 0.
\]
Choose a \( C^{\infty}\) cutoff \(\chi\colon[0,\infty)\to[0,1]\), which we need since we are actually differentiating, with \(\chi=1\) on \([0,1]\) and
\(\chi=0\) on \([2,\infty)\), and define \[
  R_i^r(u,w):=\chi(\|w\|/r)R_i(u,w).
\]
Then \[
  \mathcal R_u^r(t,w):=R_0^r(\theta_tu,w)+\sum_{j=1}^m u_j(t)R_j^r(\theta_tu,w)
\]
is a global control-affine Carath\'eodory source, with Lipschitz constant \[
  \delta_r\to0\qquad(r\to0).
\]
For \(r>0\) small enough, \[
  \delta_r<\frac{\delta}{2D}.
\]

The cutoff system is finite-dimensional in the fibers and globally bounded and
Lipschitz after cutoff, hence it admits a globally defined cutoff flow. Its
continuity again follows from Kawan's weak-* continuity theorem, since the
cutoff source remains control-affine with jointly continuous coefficients. Applying
the global Hartman--Grobman theorem to the cutoff system gives
\(\eta_r\in\mathcal B\) and \[
  H_0:=\id+\eta_r
\]
such that \[
  H_0(\Phi_t(u,v))=\widetilde\Psi_t^r(H_0(u,v)).
\]
Moreover, since \(R_i(u,0)=0\) and \(0\le\chi\le1\), \(\|R_i^r(u,w)\|\le 2rL(r)\) for all \(i\), \(u\), \(w\), so the cutoff source is bounded by \(c\,rL(r)\) with \(c:=2\bigl(1+m\max_{v\in U}|v|_\infty\bigr)\). As in the proof of \cref{thm:local-HG}, the estimate in the proof of \cref{lem:F-welldef} gives \[
  \|\eta_r\|_{\mathcal B}\le\frac{2Dc\,rL(r)}{\delta}
  \qquad\text{so}\qquad \frac{\|\eta_r\|_{\mathcal B}}{r}\to0\qquad(r\to0),
\]
and hence \[
  \rho:=r-\|\eta_r\|_{\mathcal B}>0
\]
for \(r\) sufficiently small.

Since the cutoff source is still control-affine, \cref{prop:control-continuity-layer}
gives \[
  F^r(\mathcal C)\subseteq\mathcal C,\qquad \zeta^r\in\mathcal C.
\]
Thus, \(H_0\) and \(H_0^{-1}\) are continuous bundle maps.

Restrict to \[
  V:=\{(u,v):\|v\|<\rho\}.
\]
If \(\Phi_s(u,v)\in V\), then \[
  \|H_0(\Phi_s(u,v))\|\leq \|\Phi_s(u,v)\|+\|\eta_r\|_{\mathcal B}<r.
\]
On \(\mathcal{E}_\sigma(r)\) the cutoff is identically one, so the cutoff tubular flow and
the original tubular flow coincide along the considered orbit segment. Hence, \[
  H_0(\Phi_t(u,v))=\widetilde\Psi_t(H_0(u,v)).
\]

Finally transport the conjugacy back to the manifold: \[
  H(u,v):=(u,\exp_{\sigma(u)}(H_0(u,v))).
\]
Define the open set \[
  W:=H(V).
\]
The exponential chart is a homeomorphism on the chosen tube, and \(H_0\) is a
fiberwise homeomorphism with continuous inverse. Hence, \(H:V\to W\) is a
bundle-respecting homeomorphism with continuous inverse. The definition of
\(\widetilde\Psi\) gives \[
  H(\Phi_t(u,v))=\Psi_t(H(u,v))
\]
along every orbit segment staying in \(V\). The strong ED and NUED assertions
follow by applying the H\"older and Lyapunov-bundle versions before transporting
by the exponential chart.
\end{proof}

\begin{corollary}[Case of Euclidean space]
\label{cor:control-flat}
Let \(M=\R^n\), assume \(F_i(0)=0\) for \(i\in \{0,1,\ldots,m\} \), and take \(\sigma =0\). Then \[
  \mathcal{E}_\sigma=\mathcal U\times\R^n,
\]
and the connection is the directional derivative. The derivative cocycle is the usual
bilinear Carath\'eodory system \[
  \dot y=\left(A_0+\sum_{j=1}^m u_j(t)A_j\right)y,\qquad A_j:=dF_j(0).
\]
The tubular source reduces to \[
  \widetilde{\mathcal R}_u(t,w)=R_0(w)+\sum_{j=1}^m u_j(t)R_j(w),
\]
where \[
  R_j(w):=F_j(w)-F_j(0)-A_jw.
\]
If this bilinear cocycle admits an exponential dichotomy uniformly over \(\mathcal U\), then the nonlinear control flow is locally fiberwise conjugate to its bilinearization.
\end{corollary}

\begin{remark}[Why affineness in the control is used]
\label{rem:control-affineness}
The weak-* continuity argument uses affineness in the present control value exactly in the pairing step. If nonlinear expressions such as \(u(t)^2\) occur, weak-* convergence does not preserve them: rapidly oscillating controls may satisfy \(u_n\rightharpoonup^\ast0\) while \(u_n^2\equiv1\). Thus, nonlinear dependence on the present control value cannot be treated on the weak-* \(L^\infty\)-hull by this argument.
\end{remark}

In contrast to the result of Baratchart, Chyba, and Pomet~\cite{BaratchartEtAl2007GrobmanHartmana}, our theorem is explicitly built to allow for non-autonomous linear parts. This ultimately allows us to linearize along arbitrary solution curves, which inherently introduces non-autonomy.

\subsection{Recovering classical results}

We want to round out this exposition by showing how to recover previous Hartman--Grobman results from our setting and illustrate how our approach further simplifies those. 

As a prime example we show how to recover the Hartman--Grobman type result obtained by Aulbach and Wanner~\cite{AulbachWanner2000HartmanGrobman}. 

Assume we are given a \textit{Carath\'eodory} differential equation on a Banach space \( X\), i.e. \begin{equation}\label{eq:NAODE}
    \dot{x} = A(t)x + f(t,x)
,\end{equation} where \( A : \R \to \mathcal{L}(X)\) is locally integrable and \( f : \R \times X \to X\) is \textit{Carath\'eodory} in the sense that for each \( t \in \R\) the map \( x \mapsto f(t,x)\) is continuous and for \( x \in X\) the map \( t \mapsto f(t,x)\) is strongly measurable. Denote by \(\widetilde{\Phi}\) (resp.\ \( \widetilde{\Psi}\)) the solution operators of the linear (resp.\ non-linear)  system. 

\begin{remark}
\label{rem:generators}
No hypothesis of \cref{sec:framework,sec:ED,sec:HGT} refers to a generator; all are formulated for the propagator \(\Phi\). For \(A\in L^1_{\mathrm{loc}}(\R,\mathcal L(X))\) the evolution operator satisfies \(\|\widetilde\Phi(t,s)\|\leq\exp\bigl|\int_s^t\|A(\sigma)\|\,d\sigma\bigr|\), each \(\widetilde\Phi(t,s)\) is invertible with inverse \(\widetilde\Phi(s,t)\), and \((t,s,x)\mapsto\widetilde\Phi(t,s)x\) is continuous; hence \(\Phi\) as defined below is a continuous linear skew-product flow in the sense of \cref{def:linearSPF}, and neither continuity nor differentiability of \(t\mapsto A(t)\) is used anywhere.
\end{remark}

Furthermore, assume the linear system admits an exponential dichotomy with a continuous splitting \( P=1-Q\) and dichotomy estimates \begin{align}
 \label{eq:ODE-ExpDich1} \| P_t\widetilde{\Phi}(t,s)\| &\leq K e^{-\delta(t-s)}   & (t \geq s),\\
 \label{eq:ODE-ExpDich2} \| Q_t\widetilde{\Phi}(t,s) \| &\leq K e^{\delta(t-s)} & (t \leq s),
\end{align} 
for constants \( K \geq 1\), \( \delta>0\).

In order to apply our result we have to translate the setting to the language of Banach bundles. This construction is well-known and part of the original motivation to study skew-product flows, see~\cite{SackerSell1978SpectralTheory}. Given \( \widetilde{\Phi} : \R \times \R \times X \to X\) and \( \widetilde{\Psi} : \R \times \R \times X \to X\), the respective linear and non-linear evolution operators associated to \cref{eq:NAODE}, and \( p : \R \times  X \to \R\), the trivial Banach bundle with fiber \( X\) over \( \R\), we now associate to \cref{eq:NAODE} the skew-product given by \[
    \theta: \R \times \R \to \R, \quad (t,s) \mapsto t+s 
\] and \( \Phi: \R \times \R \times X \to \R \times X\) given by \[
    \Phi(t,s,x) := (\theta_t s, \widetilde{\Phi} (t+s,s,x))
\] and \( \Psi : \R \times \R \times X \to \R \times  X\) by \[
    \Psi(t,s,x) := (\theta_t s, \widetilde{\Psi} (t+s,s,x))
.\] And by the solution properties of \( \widetilde{\Phi} \) and \( \widetilde{\Psi} \) they form skew-product flows on \( \R \times X \to \R\) over \( \theta\). In fact, this skew-product flow admits an exponential dichotomy in the sense of \cref{def:ED} with data \( (P,K,\delta)\). This construction now recovers results from~\cite{AulbachWanner2000HartmanGrobman}.

\begin{theorem}\label{thm:ClassicHG}
Let \(\widetilde\Phi\), as defined above, admit an exponential dichotomy with data \((P,K,\delta)\). Let \(f:\R\times X\to X\) be of Carath\'eodory type with \(f(t,0)=0\) and \begin{align}
  \|f(t,x)\| &\leq M, \\
  \|f(t,x)-f(t,y)\| & \leq L\|x-y\|  
\end{align}
for all \(t\in\R\) and \(x,y\in X\). If \(L<\delta/(2K)\), then \(\dot x=A(t)x+f(t,x)\) is topologically equivalent to \(\dot x=A(t)x\) in the sense of \cref{thm:HG}, by a continuous family of homeomorphisms \(H_s=\id+\eta_s\) with \(\sup_{s,x}\|\eta_s(x)\|<\infty\).
\end{theorem}

\begin{proof}
  By the above construction, we have that \( \Phi\) is a continuous linear skew-product flow admitting an exponential dichotomy \( (P,K,\delta)\) and \( R_s(t,x) := f(s+t,x)\) is a Carath\'eodory perturbation. Choose now \( \delta_0\) with \( L < \delta_0 < \frac{\delta}{2K}\) and renorm \( X\) with \( \lambda \|\cdot \| =: \|\cdot \|_{\lambda}  \) with \( \lambda>0\) so small that \( 2 \lambda M \leq \delta_0\). In this new norm, which leaves the dichotomy constants unaffected, we have \[
      \|R_s(t,x)-R_s(t,y)\|_{\lambda} \leq \delta_0 \min \{1, \|x-y\|_{\lambda}\}  
  .\] 
  Finally, we need to check that the additional continuity assumption \( F(\mathcal{C}) \subset \mathcal{C}\) of \cref{thm:HG} is satisfied. Since sections are now simply functions \( \eta: \R \to X\), the proof simplifies substantially from the general case. 
  For \(\eta\in\mathcal C\), the Green operator defining \(F\eta\) has integrands depending continuously on the linear evolution, on \(\eta\), and on the phase variable.  The only merely measurable dependence is the translated Carathéodory time dependence \(t\mapsto f(s+t,\cdot)\). The global boundedness and Lipschitz constant of the perturbation, the continuity of translations in \( L^1_{\mathrm{loc}}\) and the exponential dichotomy estimate give an integrable dominating function of the integrals occurring in the fixed-point operator. An application of the dominated convergence theorem, similar to \cref{prop:ContForContBunMap,prop:control-continuity-layer}, gives the joint continuity of \[
    (s,x) \mapsto (F \eta)_s(x)
.\] Identical considerations lead to the continuity of the inverse conjugacy. Thus, we recover the continuously varying conjugacies.
  Hence, \cref{thm:HG} applies in its generality.  
\end{proof}

\begin{remark}[Projected variant]
\label{rem:projected-classical-variant}
The proof of \cref{thm:HG} only uses the perturbation through the projected terms \(PR\) and \(QR\) in the two Green integrals. Hence, in the classical trivial-bundle setting, the same proof remains valid if the full min-bound is replaced by \[
  \|P(t)(f(t,x)-f(t,y))\|, \ \|Q(t)(f(t,x)-f(t,y))\| \leq \delta_0\min\{1,\|x-y\|\}.
\]
This projected version will be used in \cref{cor:AW}. Whereas Aulbach--Wanner are working on a system with a decoupled linear part \( A = \mathrm{diag}(A_1,A_2)\) with a constant projection, we need no such assumptions, but are able to recover their results.
\end{remark}

\begin{corollary}\label{cor:AW}
\cite[Theorem~3.1]{AulbachWanner2000HartmanGrobman} follows from the projected variant in \cref{rem:projected-classical-variant}. Precisely, let \(X=X_1\oplus X_2\) carry the \(\ell^1\)-norm, let \(A=\operatorname{diag}(A_1,A_2)\) satisfy hypothesis \textup{(A1)} of~\cite{AulbachWanner2000HartmanGrobman} with \(K\ge1\) and \(\alpha<0<\beta\), and let \(f=(f_1,f_2)\) satisfy hypothesis \textup{(A2)} of~\cite{AulbachWanner2000HartmanGrobman} with constants \(L\) and \(M\). Put \[
  P(t)(x_1,x_2):=(x_1,0), \qquad Q(t)(x_1,x_2):=(0,x_2),
\]
and \[
  \gamma:=\min\{-\alpha,\beta\}.
\]
Then the associated linear skew-product flow has an exponential dichotomy with data \((P,K,\delta)\) for every \[
  0<\delta<\gamma.
\]
Moreover, \textup{(A2)} gives the projected Lipschitz estimates \[
  \|P(t)(f(t,x)-f(t,y))\|_1\le L\|x-y\|_1,
\]
and \[
  \|Q(t)(f(t,x)-f(t,y))\|_1\le L\|x-y\|_1.
\]
Hence, the projected Hartman--Grobman theorem, outlined in \cref{rem:projected-classical-variant}, applies whenever \[
  L<\frac{\gamma}{2K},
\]
which is implied by and improves the Aulbach--Wanner smallness condition \[
  0\leq L< \frac{\delta_{\mathrm{AW}}}{2K^2} \bigl(K+2-\sqrt{K^2+4}\bigr), \qquad \alpha<-\delta_{\mathrm{AW}}<0<\delta_{\mathrm{AW}}<\beta .
\]
\end{corollary}

\begin{proof}
Choose \(\delta \in(0,\gamma)\) with \(2KL<\delta\), and then choose \(\delta_0\) such that \[
  L<\delta_0<\frac{\delta}{2K}.
\]
Since Aulbach--Wanner assume \(\|f_i(t,x)\| \leq M\) componentwise, we have \[
  \|P(t)f(t,x)\|_1,\ \|Q(t)f(t,x)\|_1\le M.
\]
After the scalar renorming \(\|x\|_\lambda=\lambda\|x\|_1\), with \(2\lambda M\leq\delta_0\), the projected estimates become \[
  \|P(t)(f(t,x)-f(t,y))\|_\lambda \leq \delta_0\min\{1,\|x-y\|_\lambda\},
  \] analogously for \( Q\).
Thus, the projected variant, \cref{rem:projected-classical-variant}, applies.

If the coefficients are \(\omega\)-periodic, the same argument is applied over the quotient base \( \R / \omega \Z\), yielding an \(\omega\)-periodic conjugacy.  If the system is autonomous, the base space is a single point, and the conjugacy is time-independent. Hence, the autonomy and periodicity conclusions follow directly from the construction. 

It remains to identify the stable and unstable manifolds. Define \[
  S(\tau):=\{\xi\in X:\widetilde\Psi(\cdot,\tau,\xi) \text{ is bounded on }[\tau,\infty)\},
\]
and \[
  U(\tau):=\{\xi\in X:\widetilde\Psi(\cdot,\tau,\xi) \text{ is bounded on }(-\infty,\tau]\}.
\] By the integral-manifold theorem of Aulbach--Wanner~\cite{aulbach_integral_1996}, these bounded-orbit sets coincide with the stable and unstable integral manifolds. 
Since \(H_t=\id+\eta_t\) and \(H_t^{-1}=\id+\zeta_t\), with \[
  \sup_{t,x}\|\eta_t(x)\|<\infty, \qquad \sup_{t,x}\|\zeta_t(x)\|<\infty,
\]
boundedness of an orbit is preserved under the conjugacy. Indeed, if \[
  \mu(t):=\widetilde\Psi(t,\tau,\xi), \qquad \nu(t):=H_t^{-1}\mu(t),
\]
then \(\nu\) is the corresponding linear orbit and \[
  \|\nu(t)\|\le \|\mu(t)\|+\sup\|\zeta\|, \qquad \|\mu(t)\|\le \|\nu(t)\|+\sup\|\eta\|.
\]
Thus, \(\mu\) is bounded on a half-line if and only if \(\nu\) is bounded on the same half-line.

For the linear system, the exponential dichotomy gives \[
  \{\xi:\widetilde\Phi(\cdot,\tau)\xi \text{ is bounded on }[\tau,\infty)\}=\mathrm{im}P(\tau),
\]
and \[
  \{\xi:\widetilde\Phi(\cdot,\tau)\xi \text{ is bounded on }(-\infty,\tau]\}=\mathrm{im}Q(\tau).
\]
Consequently, \[
  H_\tau^{-1}(S(\tau))=\mathrm{im}P(\tau), \qquad H_\tau^{-1}(U(\tau))=\mathrm{im}Q(\tau).
\]
In particular in the decoupled Aulbach--Wanner setting this reads \[
  \mathrm{im}P(\tau)=X_1\times\{0\}, \qquad \mathrm{im}Q(\tau)=\{0\}\times X_2.
\]
Hence, \(H_\tau^{-1}\) maps the nonlinear stable and unstable sets onto the corresponding linear subspaces, equivalently \(H_\tau\) maps the linear stable and unstable subspaces onto the nonlinear ones. 
\end{proof}

\begin{remark}\label{rem:AW-variants}
  While recovering the statements obtained by Aulbach and Wanner in the case of Carath\'eodory differential equations, we also obtain naturally a nonuniform analog of the theorem by applying \cref{thm:NU-HG-continuous}. At the same time, by slightly strengthening the hypotheses we would also obtain Hölder regularity of the conjugacies by applying \cref{thm:Holder-HG} and \cref{thm:NU-HG-Holder}. A local version is, of course, also possible by applying \cref{thm:local-HG}.
\end{remark}

\section*{Acknowledgements}
The authors thank Henrik Kreidler for pointing us to the works of Kreidler and Siewert~\cite{kreidler_gelfand-type_2020,Siewert2020WeightedKoopman} and to adjacent works~\cite{FellDoran1988Representations,gierz_bundles_1982}, which have proven immensely valuable to our investigation.

\section*{Declaration of generative AI and AI-assisted technologies in the manuscript preparation process}
During the preparation of this work the authors used large language models in order to simplify and adapt proofs, to create outlines of early drafts based on the authors' own preexisting work, to check proofs, computations and references, and to revise the language and the \LaTeX{} source.

All results and arguments were verified and, where necessary, edited by the authors, who take full responsibility for the content of the published article.

\bibliographystyle{amsplain}
\bibliography{references}

\end{document}